\documentclass[a4paper]{amsart}

\usepackage[mathscr]{eucal}
\usepackage{amsmath,amssymb,amsbsy,amsthm,amsfonts,comment}
\usepackage[all]{xy}
\usepackage{tikz}
\usetikzlibrary{positioning,decorations.pathreplacing}
\usepackage{hyperref}
\hypersetup{
  pdfauthor = {Ryuma Orita},
  pdfkeywords = {Symplectic manifolds, non-contractible periodic orbits, Hamiltonian flows, Floer homology},
  pdftitle = {On the existence of infinitely many non-contractible periodic orbits of Hamiltonian diffeomorphisms of closed symplectic manifolds},
  pdfpagemode = UseNone,
  bookmarksnumbered=true,
}

\title[Infinitely many non-contractible periodic orbits]%
{On the existence of infinitely many non-contractible periodic orbits of Hamiltonian diffeomorphisms of closed symplectic manifolds}

\author{Ryuma Orita} 
\address{National Center for Theoretical Sciences, Taipei 10617, Taiwan}
\email{ryuma.orita@gmail.com}
\urladdr{https://ryuma-orita.github.io/}

\thanks{This work was supported by JSPS KAKENHI Grant Number JP02607057 and the Program for Leading Graduate Schools, MEXT, Japan.
The author was supported by the Grant-in-Aid for JSPS fellows.}

\newtheorem{theorem}{Theorem}[section]
\newtheorem{lemma}[theorem]{Lemma}
\newtheorem{proposition}[theorem]{Proposition}
\newtheorem{corollary}[theorem]{Corollary}

\theoremstyle{definition}
\newtheorem{definition}[theorem]{Definition}

\theoremstyle{remark}
\newtheorem{remark}[theorem]{Remark}

\newcommand{\Cr}{\mathop{\mathrm{Crit}}\nolimits}

\newcommand{\rank}{\mathop{\mathrm{rank}}\nolimits}

\newcommand{\Ker}{\mathop{\mathrm{Ker}}\nolimits}
\newcommand{\Image}{\mathop{\mathrm{Im}}\nolimits}

\newcommand{\gap}{\mathop{\mathrm{gap}}\nolimits}
\newcommand{\sgn}{\mathop{\mathrm{sgn}}\nolimits}

\newcommand{\relmiddle}[1]{\mathrel{}\middle#1\mathrel{}}

\begin{document}

\begin{abstract}
We show that the presence of a non-contractible one-periodic orbit of a Hamiltonian diffeomorphism of a connected closed symplectic manifold $(M,\omega)$ implies the existence of infinitely many non-contractible simple periodic orbits,
provided that the symplectic form $\omega$ is aspherical and the fundamental group $\pi_1(M)$ is either a virtually abelian group or an $\mathrm{R}$-group.
We also show that a similar statement holds for Hamiltonian diffeomorphisms of closed monotone or negative monotone symplectic manifolds under the same conditions on their fundamental groups.
These results generalize some works by Ginzburg and G\"urel.
The proof uses the filtered Floer--Novikov homology for non-contractible periodic orbits.
\end{abstract}

\maketitle

\tableofcontents


\section{Introduction}

Let $(M,\omega)$ be a connected closed symplectic manifold and $H\colon S^1\times M\to\mathbb{R}$ a Hamiltonian on $M$.
The Hamiltonian $H$ defines the Hamiltonian isotopy $\{\varphi_H^t\}_{t\in\mathbb{R}}$ with $\varphi_H^0=\mathrm{id}$
and the Hamiltonian diffeomorphism $\varphi_H=\varphi_H^1$.
In the present paper, we study periodic orbits of the Hamiltonian isotopies of various periods.

It is one of the most important problems in symplectic geometry to find periodic solutions of Hamiltonian systems.
In 1984, Conley \cite{Co} conjectured that every Hamiltonian diffeomorphism of tori $\mathbb{T}^{2n}$ has infinitely many simple periodic orbits.
This conjecture was proved in \cite{Hi,Ma}.
Other than for tori, Ginzburg and G\"urel \cite{GG17} proved the Conley conjecture for a broad class of closed symplectic manifolds containing
closed symplectic manifolds whose first Chern class is aspherical and closed negative monotone symplectic manifolds
(see Subsection \ref{subsection:conventions} for the definitions).

The Conley conjecture fails for the 2-sphere $S^2$.
Indeed, an irrational rotation on $S^2$ about the $z$-axis is a Hamiltonian diffeomorphism having no periodic points except two fixed points.
However, Franks \cite{Fr92,Fr96} proved that every Hamiltonian diffeomorphism of $S^2$ with at least three fixed points
has infinitely many simple periodic orbits.
Concerning this phenomenon,
Hofer and Zehnder \cite[Chapter 6]{HZ} conjectured that
every Hamiltonian diffeomorphism with more non-degenerate fixed points than a lower bound derived from the Arnold conjecture
has infinitely many simple periodic orbits.

G\"urel \cite{Gu13} proved a theorem along this line, where the threshold is the existence of a non-contractible non-degenerate (or just homologically non-trivial) one-periodic orbit.
For closed symplectic manifolds, non-contractible periodic orbits are unnecessary
in the sense that the total Floer homology $\mathrm{HF}(H;\alpha)$ for non-contractible periodic orbits representing $\alpha\neq 0$ always vanishes.
Actually, she \cite{Gu13} proved that every Hamiltonian diffeomorphism $\varphi_H$ of a closed symplectic manifold
equipped with an atoroidal (see Subsection \ref{subsection:conventions} for the definition) symplectic form has infinitely many simple periodic orbits,
provided that $\varphi_H$ has a non-contractible homologically non-trivial one-periodic orbit (see also \cite[Theorem 2.4]{GG16} for a refined version of her theorem).
To be more precise, she proved

\begin{theorem}[{\cite[Theorem 1.1]{Gu13},\cite[Theorem 2.4]{GG16}}]\label{theorem:Gu13}
Let $(M,\omega)$ be a closed symplectic manifold with atoroidal $\omega$.
Let $H\colon S^1\times M\to\mathbb{R}$ be a Hamiltonian having a non-degenerate one-periodic orbit $x$ in the homotopy class $\alpha$
such that $[\alpha]\neq 0$ in $H_1(M;\mathbb{Z})/\mathrm{Tor}$,
and $\mathcal{P}_1(H;[\alpha])$ is finite.
Then for every sufficiently large prime $p$,
the Hamiltonian $H$ has a simple periodic orbit in the homotopy class $\alpha^p$
and with period either $p$ or its next prime $p'$.
Moreover, if $\pi_1(M)$ is torsion-free hyperbolic, then the condition $[\alpha]\neq 0$ can be replaced by $\alpha\neq 1$,
i.e., $\alpha$ not being the connected components of contractible loops,
and no finiteness condition is needed.
\end{theorem}

Here $\mathcal{P}_1(H;[\alpha])$ is the set of one-periodic orbits of $\varphi_H$ representing the homology class $[\alpha]$.
The author \cite[Theorem 1.1]{Or} proved that the conclusion of Theorem \ref{theorem:Gu13} holds for the tori $(\mathbb{T}^{2n},\omega_{\mathrm{std}})$.
We note that the standard symplectic form $\omega_{\mathrm{std}}$ on $\mathbb{T}^{2n}$ is not atoroidal but aspherical.
It is worth pointing out here that Theorem \ref{theorem:Gu13} implies the existence of infinitely many \textit{non-contractible} simple periodic orbits of $\varphi_H$.
Focusing on non-contractible ones,
Ginzburg and G\"urel \cite{GG16} proved that a statement similar to Theorem \ref{theorem:Gu13} holds
for closed toroidally monotone or toroidally negative monotone (see Subsection \ref{subsection:conventions} for the definition) symplectic manifolds under an assumption on the ``Euler characteristic" $\chi$.
More precisely, they proved

\begin{theorem}[{\cite[Theorem 2.2]{GG16}}]\label{theorem:GG16}
Let $(M,\omega)$ be a closed toroidally monotone or toroidally negative monotone symplectic manifold.
Let $H\colon S^1\times M\to\mathbb{R}$ be a Hamiltonian
such that $\mathcal{P}_1(H;[\alpha])$ is finite,
and $\chi(H,I;\alpha)\neq 0$ for some interval $I$ with $\partial I\cap\widetilde{\mathrm{Spec}}(H;\alpha)=\emptyset$,
where $\alpha\in [S^1,M]$, $[\alpha]\neq 0$ in $H_1(M;\mathbb{Z})/\mathrm{Tor}$.
Then for every sufficiently large prime $p$, the Hamiltonian $H$ has a simple periodic orbit in the homotopy class $\alpha^p$
and with period either $p$ or its next prime $p'$.
Moreover, if $\pi_1(M)$ is torsion-free hyperbolic, then the condition $[\alpha]\neq 0$ can be replaced by $\alpha\neq 1$
and no finiteness condition is needed.
\end{theorem}

Here $\chi(H,I;\alpha)$ is the sum of the Poincar\'e--Hopf indices
of the Poincar\'e return maps of one-periodic orbits of $\varphi_H$ representing $\alpha$ with augmented action (see Subsection \ref{subsection:augmentedaction} for the definition) in $I$,
and $\widetilde{\mathrm{Spec}}(H;\alpha)$ is the set of values of the augmented action of one-periodic orbits of $\varphi_H$ representing $\alpha$.


\section{Main results}\label{section2}

Let us now state our main results.
Let $(M,\omega)$ be a connected closed symplectic manifold.
Let $\alpha\in [S^1,M]=\pi_1(M)/\sim_{\mathrm{conj}}$ be a free homotopy class of loops in $M$ and choose $\gamma_{\alpha}\in\pi_1(M)$ whose conjugacy class is $\alpha$.

\subsection{Results}

If $\pi_1(M)$ is \textit{virtually abelian}, by definition,
it contains an abelian subgroup $A$ of finite index.
Since $(\pi_1(M):A)<\infty$, there exists $\ell_{\alpha}\in\{1,\ldots,(\pi_1(M):A)\}$ such that $\gamma_{\alpha}^{\ell_{\alpha}}\in A$.
Let $q_{\alpha}$ be an arbitrary positive integer coprime to $\ell_{\alpha}$.
We consider the set $P_{q_{\alpha},\ell_{\alpha}}$ of primes congruent to $q_{\alpha}$ modulo $\ell_{\alpha}$
\begin{align}\label{eq:pl}
	P_{q_{\alpha},\ell_{\alpha}}
	&=\{\,p\in\mathbb{N}\mid p\ \text{is prime},\ p\equiv q_{\alpha}\ \text{mod}\ \ell_{\alpha}\,\}\\
	&=\{\,p_i\mid i\in\mathbb{N},\ p_i<p_{i+1}\,\}.\notag
\end{align}
Dirichlet's theorem on arithmetic progressions \cite{Di} asserts that $\#P_{q_{\alpha},\ell_{\alpha}}=\infty$.

On the other hand, if $\pi_1(M)$ is an \textit{$\mathrm{R}$-group}
(i.e., a group that the equality $g^n=h^n$ always implies $g=h$, where $g,h\in\pi_1(M)$ and $n\in\mathbb{N}$),
we think of $\ell_{\alpha}$ as an arbitrary positive integer.

One of our main results is the following theorem.

\begin{theorem}\label{theorem:main}
Assume that $\omega$ is aspherical and $\pi_1(M)$ is either a virtually abelian group or an $\mathrm{R}$-group.
Let $H\colon S^1\times M\to\mathbb{R}$ be a Hamiltonian
having a non-degenerate one-periodic orbit $x$ in the homotopy class $\alpha$
such that $[\alpha]\neq 0$ in $H_1(M;\mathbb{Z})/\mathrm{Tor}$,
$\mathcal{P}_1(H;[\alpha])$ is finite and $\omega$ is $\alpha$-toroidally rational
$($see Subsection \ref{subsection:conventions}$)$.
Then for every sufficiently large prime $p_i\in P_{q_{\alpha},\ell_{\alpha}}$,
the Hamiltonian $H$ has a simple periodic orbit in the homotopy class $\alpha^{p_i}$
and with period either $p_i$ or $p_{i+1}$.
Moreover, when $\pi_1(M)$ is an $\mathrm{R}$-group, then the finiteness condition on $\mathcal{P}_1(H;[\alpha])$
can be replaced by that on $\mathcal{P}_1(H;\alpha)$.
\end{theorem}

When $\pi_1(M)$ is an $\mathrm{R}$-group,
we can also prove that for every sufficiently large prime $p$,
the Hamiltonian $H$ has a simple periodic orbit in $\alpha^p$ and with period either $p$ or its next prime $p'$.

The main tool for the proof of Theorem \ref{theorem:main} is the filtered Floer--Novikov homology $\mathrm{HFN}^I(H;\alpha)$
for non-contractible periodic orbits used in \cite{Or}.
The main difficulty in using the Floer--Novikov homology in our setting is that all lifts of orbits shifted by the Novikov actions appear as generators.
However, if $\omega$ is aspherical and $\pi_1(M)$ is either a virtually abelian group or an $\mathrm{R}$-group,
then Lemmas \ref{lemma:key} and \ref{lemma:keynilp} enable us to deal with them.

In the present paper, we also prove the following theorems which generalize Theorem \ref{theorem:GG16} under conditions on the fundamental group.

\begin{theorem}\label{theorem:main2}
Assume that $(M,\omega)$ is monotone or negative monotone with monotonicity constant $\lambda$ and $\pi_1(M)$ is virtually abelian.
Let $H\colon S^1\times M\to\mathbb{R}$ be a Hamiltonian
such that $\mathcal{P}_1(H;[\alpha])$ is finite,
and $\chi(H,I;\alpha)\neq 0$ for every sufficiently small interval $I$ centered at some $s\in\widetilde{\mathrm{Spec}}(H;\alpha)$,
where $\alpha\in [S^1,M]$, $[\alpha]\neq 0$ in $H_1(M;\mathbb{Z})/\mathrm{Tor}$
and $\omega$ is $\alpha$-toroidally rational.
Then for every sufficiently large prime $p_i\in P_{q_{\alpha},\ell_{\alpha}}$, the Hamiltonian $H$ has a simple periodic orbit in the homotopy class $\alpha^{p_i}$
and with period either $p_i$ or $p_{i+1}$.
\end{theorem}

If $\pi_1(M)$ is an $\mathrm{R}$-group,
then we can relax the condition on $\chi(H,I;\alpha)$ as follows:

\begin{theorem}\label{theorem:main2nilp}
Assume that $(M,\omega)$ is monotone or negative monotone with monotonicity constant $\lambda$ and $\pi_1(M)$ is an $\mathrm{R}$-group.
Let $H\colon S^1\times M\to\mathbb{R}$ be a Hamiltonian
such that $\mathcal{P}_1(H;\alpha)$ is finite,
and $\chi(H,I;\alpha)\neq 0$ for some interval $I=[a,b)$ with $a,b\in\mathbb{R}\setminus\widetilde{\mathrm{Spec}}(H;\alpha)$,
where $\alpha\in [S^1,M]$, $[\alpha]\neq 0$ in $H_1(M;\mathbb{Z})/\mathrm{Tor}$
and $\omega$ is $\alpha$-toroidally rational.
Then for every sufficiently large prime $p$, the Hamiltonian $H$ has a simple periodic orbit in the homotopy class $\alpha^p$
and with period either $p$ or its next prime $p'$.
\end{theorem}

As in Theorem \ref{theorem:main}, one can show that
for any pair of coprime positive integers $(q_{\alpha},\ell_{\alpha})$ and every sufficiently large prime $p_i\in P_{q_{\alpha},\ell_{\alpha}}$,
the Hamiltonian $H$ has a simple periodic orbit in the homotopy class $\alpha^{p_i}$
and with period either $p_i$ or $p_{i+1}$.

For the proof, we review the augmented action filtration on the Floer--Novikov homology $\widetilde{\mathrm{HFN}}^I(H;\alpha)$ introduced in \cite{GG09,GG16}.
We note that if $(M,\omega)$ is toroidally monotone or toroidally negative monotone as in \cite{GG16},
then the augmented action does not depend on the choice of the capping.
However in our setting, it does.

\subsection{Examples}

One important example for Theorem \ref{theorem:main} is the tori $\mathbb{T}^{2n}$ with the standard symplectic form.
This Theorem \ref{theorem:main} generalizes \cite[Theorem 1.1]{Or}.
Even when $\pi_1(M)$ is just finitely generated abelian,
we have numerous examples due to the following theorem.

\begin{theorem}[{\cite[Theorem 1.2]{KRT}}]
Let $G$ be a finitely generated abelian group.
Then there exists a closed symplectic manifold $(M,\omega)$ with aspherical $\omega$ such that $\pi_1(M)=G$
if and only if
either $G\cong\mathbb{Z}\oplus\mathbb{Z}$ or $\rank{G}\geq 4$.
\end{theorem}

Another interesting example is the Kodaira--Thurston manifold $\mathrm{KT}$,
which is the product of the circle and the Heisenberg manifold.
Namely,
\[
	\mathrm{KT}=S^1\times (H(\mathbb{R})/H(\mathbb{Z})),
\]
where $H(R)$ denotes the set of the upper triangular unipotent $3\times 3$ matrices with coefficients in a given ring $R$.
The fundamental group $\pi_1(\mathrm{KT})$ is isomorphic to $\mathbb{Z}\times H(\mathbb{Z})$,
and hence it is torsion-free nilpotent, in particular, an $\mathrm{R}$-group.
We note that $\mathrm{KT}$ naturally admits an aspherical symplectic form.

On the $\alpha$-rationality condition on $\omega$ in Theorem \ref{theorem:main},
we have the following.

\begin{proposition}[{\cite[Proposition 1.5]{IKRT}}]
Let $M$ be a closed symplectic manifold equipped with an aspherical symplectic form.
Then $M$ admits an aspherical symplectic form $\omega$ such that $\langle[\omega],a\rangle\in\mathbb{Z}$ for all $a\in H_2(M;\mathbb{Z})$.
\end{proposition}

Let us now discuss examples for Theorems \ref{theorem:main2} and \ref{theorem:main2nilp}.
Let $(N,\omega_N)$ be a connected closed symplectically aspherical
(i.e., $\omega_N$ and $c_1=c_1(N,\omega_N)$ are both aspherical) symplectic manifold
whose fundamental group is a virtually abelian group or an $\mathrm{R}$-group
(e.g., $N=\mathbb{T}^{2n}$, $\mathrm{KT}$).
Then the product $(N\times\mathbb{C}P^m,\omega_N\oplus\omega_{\mathrm{FS}})$
of $(N,\omega_N)$ and the complex projective space $\mathbb{C}P^m$ equipped with the Fubini--Study form $\omega_{\mathrm{FS}}$
satisfies the assumptions of Theorem \ref{theorem:main2} or \ref{theorem:main2nilp}.


\section{Preliminaries}

In this section, first we set conventions and notation.
Then we define the filtered Floer--Novikov homology which is the main tool for the proof of the main theorems.

\subsection{Conventions and notation}\label{subsection:conventions}

Let $X$ be a connected CW-complex.
Let $\mathcal{L}X=C(S^1,X)$ be the space of free loops in $X$ where $S^1=\mathbb{R}/\mathbb{Z}$.
For a free homotopy class $\alpha\in [S^1,X]$,
denote by $\mathcal{L}_{\alpha}X$ the component of $\mathcal{L}X$ with loops representing $\alpha$.
We choose a loop $z_{\alpha}\in\mathcal{L}X$ whose free homotopy class is $\alpha$.

Every element of $\pi_1(\mathcal{L}_{\alpha}X,z_{\alpha})$ is represented by
a map $v\colon S^1\times S^1\to X$ such that
$v|_{\{0\}\times S^1}=v|_{\{1\}\times S^1}=z_{\alpha}$.
We denote by $[S^1\times S^1]\in H_2(S^1\times S^1;\mathbb{Z})\cong\mathbb{Z}$ the fundamental class of $S^1\times S^1$.
We define a homomorphism
\[
	f\colon\pi_1(\mathcal{L}_{\alpha}X,z_{\alpha})\to H_2(X;\mathbb{Z})
\]
by $f([v])=v_{\ast}([S^1\times S^1])$,
where $v_{\ast}\colon H_2(S^1\times S^1;\mathbb{Z})\to H_2(X;\mathbb{Z})$.
Then a cohomology class $u\in H^2(X;R)$ defines a cohomology class
\[
	\overline{u}\in H^1(\mathcal{L}X;R)=%
	\mathrm{Hom}(H_1(\mathcal{L}X;\mathbb{Z}),R)=%
	\mathrm{Hom}(\pi_1(\mathcal{L}_{\alpha}X,z_{\alpha}),R)
\]
by the formula $\overline{u}=u\circ f$, where $R=\mathbb{R}$ or $\mathbb{Z}$.

A cohomology class $u\in H^2(X;R)$ is called \textit{aspherical} if $u$ vanishes on $\pi_2(X)$.
Similarly, a cohomology class $u\in H^2(X;R)$ is called \textit{atoroidal} if the cohomology class $\overline{u}$ vanishes
on $\pi_1(\mathcal{L}_{\alpha}X,z_{\alpha})$ for any $\alpha\in [S^1,X]$.
We note that every atoroidal class is aspherical.

A cohomology class $u\in H^2(X;R)$ is called \textit{$\alpha$-toroidally rational} if the set
$\langle \overline{u},\pi_1(\mathcal{L}_{\alpha}X,z_{\alpha})\rangle$ is discrete in $\mathbb{R}$.
Namely, if $u$ is $\alpha$-toroidally rational, then there exists a number $h_{\alpha}\in\mathbb{R}$ such that
\[
	\langle \overline{u},\pi_1(\mathcal{L}_{\alpha}X,z_{\alpha})\rangle=%
	h_{\alpha}\mathbb{Z}.
\]

Let $(M,\omega)$ be a connected closed symplectic manifold.
We call a closed 2-form $\eta\in\Omega^2(M)$ \textit{aspherical} (resp.\ \textit{atoroidal}, \textit{$\alpha$-toroidally rational})
if its cohomology class $[\eta]$ is aspherical (resp.\ atoroidal, $\alpha$-toroidally rational).

As is explained above, the symplectic form $\omega\in \Omega^2(M)$ and the first Chern class $c_1\in H^2(M;\mathbb{Z})$ of $(M,\omega)$ define the cohomology classes
\[
	\overline{[\omega]}\in H^1(\mathcal{L}M;\mathbb{R})=\mathrm{Hom}(H_1(\mathcal{L}M;\mathbb{Z}),\mathbb{R})
\]
and
\[
	\overline{c_1}\in H^1(\mathcal{L}M;\mathbb{Z})=\mathrm{Hom}(H_1(\mathcal{L}M;\mathbb{Z}),\mathbb{Z}),
\]
respectively.
A symplectic manifold $(M,\omega)$ is called \textit{monotone} (resp.\ \textit{negative monotone}) if we have
\[
	[\omega]|_{\pi_2(M)}=%
	\lambda c_1|_{\pi_2(M)}
\]
for some non-negative (resp.\ negative) number $\lambda\in\mathbb{R}$.
Similarly, a symplectic manifold $(M,\omega)$ is called \textit{toroidally monotone} (resp.\ \textit{toroidally negative monotone})
if there exists a non-negative (resp.\ negative) number $\lambda\in\mathbb{R}$ such that for all $\alpha\in [S^1,M]$,
\[
	\overline{[\omega]}|_{\pi_1(\mathcal{L}_{\alpha}M,z_{\alpha})}=%
	\lambda\overline{c_1}|_{\pi_1(\mathcal{L}_{\alpha}M,z_{\alpha})}
\]
holds.
We note that every toroidally monotone (resp.\ toroidally negative monotone) symplectic manifold is monotone (resp.\ negative monotone).

We note that every atoroidal symplectic form is $\alpha$-toroidally rational with $h_{\alpha}=0$ for any $\alpha\in [S^1,M]$.
Moreover, every toroidally monotone or toroidally negative monotone symplectic form is an $\alpha$-toroidally rational symplectic form
with $h_{\alpha}=\lambda c_{1,\alpha}^{\min}$ for any $\alpha\in [S^1,M]$,
where $c_{1,\alpha}^{\min}\in\mathbb{N}$ is the \textit{$\alpha$-minimal first Chern number} given by
\[
	\langle \overline{c_1},\pi_1(\mathcal{L}_{\alpha}M,z_{\alpha})\rangle=%
	c_{1,\alpha}^{\min}\mathbb{Z}.
\]
Similarly, the \textit{minimal first Chern number} $c_1^{\mathrm{min}}\in\mathbb{N}$ is given by
\[
	\langle c_1,\pi_2(M)\rangle=%
	c_1^{\min}\mathbb{Z}.
\]
We note that $c_{1,\alpha}^{\min}$ divides $c_1^{\min}$.

In the present paper, we assume that all Hamiltonians $H$ are one-periodic in time, i.e., $H\colon S^1\times M\to\mathbb{R}$,
and we set $H_t=H(t,\cdot)$ for $t\in S^1=\mathbb{R}/\mathbb{Z}$.
The \textit{Hamiltonian vector field} $X_{H_t}\in \mathfrak{X}(M)$ associated to $H_t$ is defined by
\[
	\iota_{X_{H_t}}\omega=-dH_t.
\]
The \textit{Hamiltonian isotopy} $\{\varphi_H^t\}_{t\in\mathbb{R}}$ associated to $H$ is defined by
\[
	\begin{cases}
		\varphi_H^0=\mathrm{id},\\
		\frac{d}{dt}\varphi_H^t=X_{H_t}\circ\varphi_H^t\quad \text{for all}\ t\in\mathbb{R},
	\end{cases}
\]
and its time-one map $\varphi_H=\varphi_H^1$ is referred to as the \textit{Hamiltonian diffeomorphism} generated by $H$.
For $k\in\mathbb{N}$, let $\mathcal{P}_k(H;\alpha)$ be the set of $k$-periodic (i.e., defined on $\mathbb{R}/k\mathbb{Z}$) orbits
of the Hamiltonian isotopy $\{\varphi_H^t\}_{t\in\mathbb{R}}$ representing $\alpha$.
A one-periodic orbit $x\in\mathcal{P}_1(H;\alpha)$ is called \textit{non-degenerate}
if it satisfies $\det\bigl((d\varphi_H)_{x(0)}-\mathrm{id}\bigr)\neq 0$.
Moreover, $H$ is said to be \textit{$\alpha$-regular}
if all one-periodic orbits of $H$ representing $\alpha$ are non-degenerate.

Let $K$ and $H$ be two one-periodic Hamiltonians. The composition $K\natural H$ is defined by
\[
	(K\natural H)_t=K_t+H_t\circ (\varphi_K^t)^{-1}.
\]
Then the isotopy defined by $K\natural H$ coincides with $\varphi_K^t\circ\varphi_H^t$.
For $k\in\mathbb{N}$, we set $H^{\natural k}=H\natural\cdots\natural H$ ($k$ times).
We denote by $x^k$ the $k$-th iteration of a one-periodic orbit $x$ of $H$.
To be more precise, $x^k$ is the $k$-periodic orbit $x\colon\mathbb{R}/k\mathbb{Z}\to M$ of $H$.
Since there is an action-preserving and mean index-preserving one-to-one correspondence between
the set of $k$-periodic orbits of $H$ and
the set of one-periodic orbits of $H^{\natural k}$,
we can think of $x^k$ as the one-periodic orbit of $H^{\natural k}$ later.


\subsection{Floer--Novikov homology}\label{subsection:floernovikov}

In this subsection, we define the Floer--Novikov homology for non-contractible periodic orbits (see, e.g., \cite{On,Or} for details).
Let $(M,\omega)$ be a $2n$-dimensional connected closed symplectic manifold.
For simplicity, we assume that $(M,\omega)$ is \textit{weakly monotone}, i.e.,
$M$ satisfies one of the following conditions; $M$ is monotone, or $c_1$ is aspherical, or $c_1^{\mathrm{min}}\geq n-2$.
We note that the Floer--Novikov homology can be defined for general symplectic manifolds due to the technique of virtual cycles \cite{FO,LT,Pa}.

Let $H\colon S^1\times M\to\mathbb{R}$ be a Hamiltonian.
For a free homotopy class $\alpha\in [S^1,M]$,
we fix a reference loop $z_{\alpha}\in\alpha$.

\subsubsection{Action functional}

We consider the universal covering space $\widetilde{\mathcal{L}_{\alpha}M}$ of $\mathcal{L}_{\alpha}M$
and define the covering space $\pi\colon\overline{\mathcal{L}_{\alpha}M}\to\mathcal{L}_{\alpha}M$ with fiber being the group
\[
	\Gamma_{\alpha}=\frac{\pi_1(\mathcal{L}_{\alpha}M,z_{\alpha})}{\Ker\overline{[\omega]}\cap\Ker\overline{c_1}}.
\]
We consider the set of pairs $(x,\Pi)$, where $x\in\mathcal{L}_{\alpha}M$
and $\Pi\colon [0,1]\times S^1\to M$ is a path in $\mathcal{L}_{\alpha}M$ joining $z_{\alpha}$ and $x$.
We set an equivalence relation $\sim$ by defining $(x_1,\Pi_1)\sim (x_2,\Pi_2)$ if and only if $x_1=x_2$,
$\langle\overline{[\omega]},\Pi_1\#(-\Pi_2)\rangle=0$
and $\langle\overline{c_1},\Pi_1\#(-\Pi_2)\rangle=0$,
where $\Pi_1\#(-\Pi_2)$ is the loop defined by the path $\Pi_1$ and the path $-\Pi_2$,
which can be seen as a toroidal 2-cycle obtained by gluing $\Pi_1$ and $\Pi_2$ with orientation reversed along the boundaries.
Then the space $\overline{\mathcal{L}_{\alpha}M}$ can be viewed as the set of such equivalence classes $[x,\Pi]$.

We define the \textit{action functional} $\mathcal{A}_H\colon \overline{\mathcal{L}_{\alpha}M}\to \mathbb{R}$ by
\[
	\mathcal{A}_H([x,\Pi])=-\int_{[0,1]\times S^1} \Pi^{\ast}\omega +\int_{0}^{1}H_t\bigl(x(t)\bigr)\, dt.
\]
Since $\pi^{\ast}\overline{[\omega]}=0\in H^1(\overline{\mathcal{L}_{\alpha}M};\mathbb{R})$,
the action functional $\mathcal{A}_H$ is well-defined.
Here we note that the critical point set $\Cr(\mathcal{A}_H)$ is equal to $\overline{\mathcal{P}}_1(H;\alpha)=\pi^{-1}\bigl(\mathcal{P}_1(H;\alpha)\bigr)$.

We fix a symplectic trivialization of $TM$ along the reference loop $z_{\alpha}$.
Then one can associate the \textit{mean index} $\Delta_H(\bar{x})$
to a capped one-periodic orbit $\bar{x}=[x,\Pi]\in\overline{\mathcal{L}_{\alpha}M}$ as follows.
By extending the trivialization of $TM|_{z_{\alpha}}$ to the capping $\Pi$, we obtain a trivialization of $TM|_x$.
Thus we get a path $t\mapsto (d\varphi_H^t)_{x(0)}$ in the group $\mathrm{Sp}(2n)$.
Now we define the mean index $\Delta_H(\bar{x})$ to be the mean index of the resulting path (see, e.g., \cite{SZ}).
Similarly, if $x$ is non-degenerate, we can define the \textit{Conley--Zehnder index} $\mu_{\mathrm{CZ}}(H,\bar{x})$ of $\bar{x}$.
We note that the above two indices have the relation
\[
	\lvert\Delta_H(\bar{x})-\mu_{\mathrm{CZ}}(H,\bar{x})\rvert\leq n,
\]
where $\dim M=2n$.
The inequality is strict when $x$ is non-degenerate.

Assume that all iterated homotopy classes $\alpha^k$, $k\in\mathbb{N}$, are distinct and non-trivial.
We choose the iterated loop $z_{\alpha}^k$ with the iterated trivialization as the reference loop for $\alpha^k$.
Then the action functional $\mathcal{A}_H$ and the mean index $\Delta_H$ are homogeneous with respect to iterations in the sense that
\[
	\mathcal{A}_{H^{\natural k}}([x,\Pi]^k)=k\mathcal{A}_H([x,\Pi])
	\quad \text{and}\quad
	\Delta_{H^{\natural k}}([x,\Pi]^k)=k\Delta_H([x,\Pi]),
\]
where $[x,\Pi]^k=[x^k,\Pi^k]$ is the $k$-th iteration of $[x,\Pi]$.
Here we think of the iterated loop $x^k$ as the loop defined on $S^1=\mathbb{R}/\mathbb{Z}$,
where $x^k$ defined on $\mathbb{R}/k\mathbb{Z}$ and on $\mathbb{R}/\mathbb{Z}$
have the same action and mean index (see, e.g., \cite[Subsection 2.1]{GG10}).
Moreover, for any $\bar{x}\in\overline{\mathcal{P}}_1(H;\alpha)$ and any $[v]\in\pi_1(\mathcal{L}_{\alpha}M,z_{\alpha})$ the equalities
\[
	\mathcal{A}_H(\bar{x}\#[v])=\mathcal{A}_H(\bar{x})-\langle\overline{[\omega]},[v]\rangle,
	\quad \Delta_H(\bar{x}\#[v])=\Delta_H(\bar{x})-2\langle\overline{c_1},[v]\rangle
\]
and
\[
	\mu_{\mathrm{CZ}}(H,\bar{x}\#[v])=\mu_{\mathrm{CZ}}(H,\bar{x})-2\langle\overline{c_1},[v]\rangle
	\quad \text{(if $x$ is non-degenerate)}
\]
hold (see, e.g., \cite[Section 2]{Ba}).
We define the \textit{action spectrum} of $\mathcal{A}_H$ by
\[
	\mathrm{Spec}(H;\alpha)=\mathcal{A}_H\bigl(\overline{\mathcal{P}}_1(H;\alpha)\bigr).
\]


\subsubsection{The filtered Floer--Novikov chain complex}

We assume that $H$ is $\alpha$-regular.
Let $J\in\mathcal{J}(M,\omega)$ be an $\omega$-compatible almost complex structure.
We consider the Floer differential equation
\begin{equation}\label{eq:1}
	\partial_s u+J(u)\bigl(\partial_t u-X_{H_t}(u)\bigr)=0
\end{equation}
for $u\colon\mathbb{R}\times S^1\to M$ where $(s,t)\in\mathbb{R}\times S^1$.
For a smooth solution $u\colon \mathbb{R}\times S^1\to M$ to \eqref{eq:1}, we define the energy by the formula
\[
	E(u)=\int_0^1\int_{-\infty}^{\infty}\lvert\partial_s u\rvert^2\,dsdt.
\]
Then we have the following:

\begin{lemma}[\cite{Sa}]\label{lemma:energy}
Let $u\colon \mathbb{R}\times S^1\to M$ be a smooth solution to \eqref{eq:1} with finite energy.
\begin{enumerate}
	\item There exist $\bar{x}^{\pm}\in\overline{\mathcal{P}}_1(H;\alpha)$ such that
		\[
			\lim_{s\to\pm\infty}u(s,t)=x^{\pm}(t)\quad \text{and}\quad \lim_{s\to\pm\infty}\partial_s u(s,t)=0,
		\]
		where $\bar{x}^+=[x^+,\Pi^+]$ and $\bar{x}^-=[x^-,\Pi^-]$, and both limits are uniform in the $t$-variable.
		Moreover, we have
		\[
			[x^+,\Pi^-\# u]=[x^+,\Pi^+]\in\widetilde{\mathcal{L}_{\alpha}M}.
		\]
	\item The energy identity holds:
		\[
			E(u)=\mathcal{A}_H(\bar{x}^-)-\mathcal{A}_H(\bar{x}^+).
		\]
\end{enumerate}
\end{lemma}

We call a family of almost complex structures \textit{regular}
if the linearized operator for \eqref{eq:1} is surjective for any finite-energy solution to \eqref{eq:1}
in the homotopy class $\alpha$.
We denote by $\mathcal{J}_{\mathrm{reg}}(H;\alpha)$ the space of regular families of almost complex structures.
$\mathcal{J}_{\mathrm{reg}}(H;\alpha)\subset\mathcal{J}(M,\omega)$ is a generic set, i.e., a set containing a countable intersection of open and dense subsets in $\mathcal{J}(M,\omega)$ (see \cite{FHS}).
For any $J\in\mathcal{J}_{\mathrm{reg}}(H;\alpha)$ and any pair $\bar{x}^{\pm}\in\overline{\mathcal{P}}_1(H;\alpha)$,
the space
\[
	\mathcal{M}(\bar{x}^-,\bar{x}^+;H,J)=\{\,\text{solution to \eqref{eq:1} satisfying (i)}\,\}
\]
is a smooth manifold, and the dimension of the connected component of any $u\in\mathcal{M}(\bar{x}^-,\bar{x}^+;H,J)$ is given by
the difference of the Conley--Zehnder indices
of $\bar{x}^-$ and $\bar{x}^+$ relative to $u$.
We denote by $\mathcal{M}^1(\bar{x}^-,\bar{x}^+;H,J)$ the subspace of solutions of relative index one.
For $J\in\mathcal{J}_{\mathrm{reg}}(H;\alpha)$, the quotient $\mathcal{M}^1(\bar{x}^-,\bar{x}^+;H,J)/\mathbb{R}$
is a finite set for any pair $\bar{x}^{\pm}\in\overline{\mathcal{P}}_1(H;\alpha)$.

Let $a$ and $b$ be real numbers such that $-\infty\leq a<b\leq\infty$ and $a,b\not\in\mathrm{Spec}(H;\alpha)$.
We set $\overline{\mathcal{P}}_1^a=\{\,\bar{x}\in\overline{\mathcal{P}}_1(H;\alpha)\mid\mathcal{A}_H(\bar{x})<a\,\}$.
We define the chain group of our Floer--Novikov chain complex to be
\[
	\mathrm{CFN}^{[a,b)}(H;\alpha)=\mathrm{CFN}^b(H;\alpha)/\mathrm{CFN}^a(H;\alpha),
\]
where
\[
	\mathrm{CFN}^a(H;\alpha)=\left\{\,\xi=\sum\xi_{\bar{x}}\bar{x}\relmiddle|%
	\begin{aligned}%
		\bar{x}\in\overline{\mathcal{P}}_1^a,\, \xi_{\bar{x}}\in\mathbb{Z}/2\mathbb{Z}\ \text{such that $\forall C\in\mathbb{R}$,}\\%
		\#\{\,\bar{x}\mid\xi_{\bar{x}}\neq 0,\, \mathcal{A}_H(\bar{x})>C\,\}<\infty%
	\end{aligned}\,\right\}.
\]

We define the boundary operator $\partial_b^{H,J}\colon \mathrm{CFN}^b(H;\alpha)\to \mathrm{CFN}^b(H;\alpha)$ by
\[
	\partial_b^{H,J}(\bar{x})=\sum \#\left(\mathcal{M}^1(\bar{x},\bar{y};H,J)/\mathbb{R}\right)\,\bar{y}
\]
for a generator $\bar{x}\in \overline{\mathcal{P}}_1^b$.

\begin{theorem}[\cite{HS}]
If $J$ is regular, then the operator $\partial_b^{H,J}$ is well-defined and satisfies $\partial_b^{H,J}\circ\partial_b^{H,J}=0$.
\end{theorem}

The energy identity (ii) in Lemma \ref{lemma:energy} implies that $\mathrm{CFN}^a(H;\alpha)$ is invariant under the boundary operator $\partial_b^{H,J}$.
Thus we get an induced operator $\partial_{[a,b)}^{H,J}$ on the quotient $\mathrm{CFN}^{[a,b)}(H;\alpha)$.

\begin{definition}
The \textit{filtered Floer--Novikov homology group} is defined to be
\[
	\mathrm{HFN}^{[a,b)}(H,J;\alpha)=\Ker{\partial_{[a,b)}^{H,J}}/\Image{\partial_{[a,b)}^{H,J}}.
\]
\end{definition}

\begin{theorem}[\cite{Fl,Sa,SZ}]
If $J_0, J_1\in\mathcal{J}(H;\alpha)$ are two regular almost complex structures, then there exists a natural isomorphism
\[
	\mathrm{HFN}^{[a,b)}(H,J_0;\alpha)\to \mathrm{HFN}^{[a,b)}(H,J_1;\alpha).
\]
\end{theorem}

We refer to $\mathrm{HFN}^{[a,b)}(H;\alpha)=\mathrm{HFN}^{[a,b)}(H,J;\alpha)$ as the Floer--Novikov homology associated to $H$.


\subsubsection{Continuation}

We define the set
\[
	\mathcal{H}^{a,b}(M;\alpha)=\{\,H\in C^{\infty}(S^1\times M)\mid a,b\not\in\mathrm{Spec}(H;\alpha)\,\}.
\]

\begin{proposition}[{\cite[Remark 4.4.1]{BPS}}]\label{proposition:nbd}
Every Hamiltonian $H\in\mathcal{H}^{a,b}(M;\alpha)$ has a neighborhood $\mathcal{U}$
such that the Floer--Novikov homology groups $\mathrm{HFN}^{[a,b)}(H',J';\alpha)$, for any $\alpha$-regular $H'\in \mathcal{U}$
and any regular almost complex structure $J'\in\mathcal{J}_{\mathrm{reg}}(H';\alpha)$, are naturally isomorphic.
\end{proposition}

Proposition \ref{proposition:nbd} enables us to define the Floer--Novikov homology $\mathrm{HFN}^{[a,b)}(H;\alpha)$
even when $H$ is not $\alpha$-regular.

\begin{definition}
For $H\in\mathcal{H}^{a,b}(M;\alpha)$, we define $\mathrm{HFN}^{[a,b)}(H;\alpha)=\mathrm{HFN}^{[a,b)}(K;\alpha)$,
where $K$ is any $\alpha$-regular Hamiltonian sufficiently close to $H$.
\end{definition}

Let $H^+$, $H^-\colon S^1\times M\to\mathbb{R}$ be two Hamiltonians.
We choose regular almost complex structures $J^{\pm}\in\mathcal{J}_{\mathrm{reg}}(H^{\pm};\alpha)$.
We consider a \textit{linear homotopy} $\{H_s\}_{s\in\mathbb{R}}$ from $H^-$ to $H^+$, i.e.,
a smooth homotopy of the form
\[
	(H_s)_t=H_t^-+\beta(s)(H_t^+-H_t^-),
\]
where $\beta\colon\mathbb{R}\to [0,1]$ is a non-decreasing function,
and choose a smooth homotopy $\{J_s\}_{s\in\mathbb{R}}$ from $J^-$ to $J^+$ such that
\[
	(H_s,J_s)=%
	\begin{cases}%
		(H^-,J^-) & \text{if\: $s\leq -R$},\\
		(H^+,J^+) & \text{if\: $s\geq R$},
	\end{cases}
\]
for some constant $R>0$.
We set $H_{s,t}=(H_s)_t$.
Let $\alpha\in [S^1,M]$ be a nontrivial free homotopy class and $a,b\in\mathbb{R}\cup\{\infty\}$
such that $a<b$ and $a,b\not\in\mathrm{Spec}(H^{\pm};\alpha)$.
It follows from the energy identity
\[
	E(u)=\mathcal{A}_{H^-}(\bar{x}^-)-\mathcal{A}_{H^+}(\bar{x}^+)+\int_0^1\int_{-\infty}^{\infty}\partial_s H\bigl(s,t,u(s,t)\bigr)\,dsdt
\]
that the Floer--Novikov chain map $\mathrm{CFN}(H^-;\alpha)\to\mathrm{CFN}(H^+;\alpha)$,
defined in terms of the solutions of the equation
\[
	\partial_s u+J_s(u)\bigl(\partial_t u-X_{H_{s,t}}(u)\bigr)=0,
\]
induces a natural homomorphism
\[
	\sigma_{H^+H^-}\colon \mathrm{HFN}^{[a,b)}(H^-;\alpha)\to \mathrm{HFN}^{[a+C,b+C)}(H^+;\alpha),
\]
where $C=C(H_s)$ is the constant given by
\[
	C=\max\left\{\int_0^1\max_{M}\left(H_t^+-H_t^-\right)dt,0\right\}
\]
(see, e.g., \cite[Subsection 4.4]{BPS}).


\section{Lemmas from algebraic topology and group theory}

In this section,
we review several necessary facts on aspherical cohomology classes,
the fundamental groups of loop spaces and elementary group theory.

\subsection{Aspherical cohomology classes and Eilenberg--MacLane spaces}

In this subsection, we collect some facts concerning aspherical cohomology classes and the Eilenberg--MacLane space.
Given a group $G$, we recall that the \textit{Eilenberg--MacLane space} $K(G,1)$ is defined to be a connected CW-complex
with fundamental group $G$ and such that $\pi_i(K(G,1))=0$ for any $i>1$.

\begin{proposition}[{\cite[Lemma 2.1]{RT}}]\label{proposition:iff}
Let $X$ be a finite CW-complex and $u\in H^2(X;\mathbb{R})$ an aspherical cohomology class.
Then for every map $f\colon X\to K(\pi_1(X),1)$ which induces an isomorphism of fundamental groups,
\[
	u\in\Image\left(f^{\ast}\colon H^2(K(\pi_1(X),1);\mathbb{R})\to H^2(X;\mathbb{R})\right).
\]
\end{proposition}

\begin{corollary}[{\cite[Lemma 4.2]{LO}}, {\cite[Corollary 2.2]{RT}}]\label{corollary:iff}
Let $(M,\omega)$ be a symplectic manifold.
Then the following conditions are equivalent.
\begin{enumerate}
	\item $\omega$ is aspherical,
	\item there exists a map $f\colon M\to K(\pi_1(M),1)$ which induces an isomorphism of fundamental groups and such that
		\[
			[\omega]\in\Image\left(f^{\ast}\colon H^2(K(\pi_1(M),1);\mathbb{R})\to H^2(M;\mathbb{R})\right),
		\]
	\item there exists a map $f\colon M\to K(\pi_1(M),1)$ such that
		\[
			[\omega]\in\Image\left(f^{\ast}\colon H^2(K(\pi_1(M),1);\mathbb{R})\to H^2(M;\mathbb{R})\right).
		\]
\end{enumerate}
\end{corollary}


\subsection{Fundamental groups of free loop spaces}

In this subsection, we describe the growth of the fundamental group of the free loop component containing iterations of a loop.
Namely, we examine how $\pi_1(\mathcal{L}_{\alpha}X)$ and $\pi_1(\mathcal{L}_{\alpha^k}X)$ differ.
Let $C_G(g)$ denote the \textit{centralizer} of an element $g$ in a group $G$:
$C_G(g)=\{\,c\in G\mid gc=cg\,\}$.
The following proposition enables us to compute the fundamental group of a component of a free loop space.

\begin{proposition}[{\cite[Proposition 1]{Ha}}]\label{proposition:centralizer}
Let $X$ be a connected topological space such that $\pi_2(X)=0$.
Let $\alpha\in [S^1,X]$ be a free homotopy class and choose $z_{\alpha}\in\mathcal{L}_{\alpha}X$ and $\gamma_{\alpha}\in\pi_1(X)$ representing $\alpha$.
Then
\[
	\pi_1(\mathcal{L}_{\alpha}X,z_{\alpha})\cong C_{\pi_1(X)}(\gamma_{\alpha}).
\]
\end{proposition}

\subsubsection{Virtually abelian groups}

From now on, we concentrate on spaces having virtually abelian fundamental groups.

\begin{definition}
A group $G$ is called \textit{virtually abelian} if it contains an abelian subgroup of finite index.
\end{definition}

Let $G$ be a virtually abelian group and $A<G$ an abelian subgroup of finite index.
For $g\in G$,
there exists $\ell_g\in\{1,\ldots,(G:A)\}$ such that $g^{\ell_g}\in A$.
Let $q_g$ be a positive integer coprime to $\ell_g$.
We prove the following useful lemma concerning virtually abelian groups.

\begin{lemma}\label{lemma:bg}
For every $k\in\mathbb{Z}_{\geq0}$ and $c\in C_G(g^{q_g+k\ell_g})$, there exists
$m\in\{1,\ldots,(G:A)\}$ such that $c^m\in C_G(g)$.
\end{lemma}

\begin{proof}
Let $k\in\mathbb{Z}_{\geq0}$ and $c\in C_G(g^{q_g+k\ell_g})\subset G$.
Then there exists $m\in\{1,\ldots,(G:A)\}$ such that $c^m\in A$.
Since $A$ is abelian, $g^{\ell_g}$ and $c^m$ commute.
Since $q_g$ and $\ell_g$ are coprime, we have
\[
	n_1q_g+n_2\ell_g=1
\]
for some $n_1$, $n_2\in\mathbb{Z}$.
Therefore, we have
\begin{align*}
	c^mgc^{-m}%
	&=c^mg^{n_1q_g+n_2\ell_g}c^{-m}=%
	c^m\left(g^{q_g+k\ell_g}\right)^{n_1}\left(g^{\ell_g}\right)^{-n_1k+n_2}c^{-m}\\
	&=\left(g^{q_g+k\ell_g}\right)^{n_1}c^mc^{-m}\left(g^{\ell_g}\right)^{-n_1k+n_2}=g.
\end{align*}
This finishes the proof.
\end{proof}

Let $X$ be a finite CW-complex whose fundamental group is virtually abelian.
Then there exists an abelian subgroup $A<\pi_1(X)$ of finite index.
Let $\alpha\in [S^1,X]$ be a free homotopy class and choose $\gamma_{\alpha}\in\pi_1(X)$ representing $\alpha$.
As above, there exists $\ell_{\alpha}\in\{1,\ldots,(\pi_1(X):A)\}$ such that $\gamma_{\alpha}^{\ell_{\alpha}}\in A$.
Let $q_{\alpha}$ be an arbitrary positive integer coprime to $\ell_{\alpha}$.

We recall that every cohomology class $u\in H^2(X;\mathbb{R})$ defines
a cohomology class $\overline{u}\in H^1(\mathcal{L}X;\mathbb{R})$ (see Subsection \ref{subsection:conventions}).
The following is the key lemma.

\begin{lemma}\label{lemma:key}
Let $X$ be a finite CW-complex whose fundamental group is virtually abelian and $u\in H^2(X;\mathbb{R})$.
Then the following conditions are equivalent.
\begin{enumerate}
	\item $u$ is aspherical,
	\item for every $\alpha\in [S^1,X]$, $k\in\mathbb{Z}_{\geq0}$ and $[v]\in\pi_1(\mathcal{L}_{\alpha^{q_{\alpha}+k\ell_{\alpha}}}X)$,
	there exist $m\in\{1,\ldots,(\pi_1(X):A)\}$ and $[w]\in\pi_1(\mathcal{L}_{\alpha}X)$ such that
		\[
			m\langle\overline{u},[v]\rangle=(q_{\alpha}+k\ell_{\alpha})\langle\overline{u},[w]\rangle,
		\]
	\item some $\alpha_0\in [S^1,X]$ satisfies the following:
	For any $k\in\mathbb{Z}_{\geq0}$ and $[v]\in\pi_1\bigl(\mathcal{L}_{\alpha_0^{q_{\alpha_0}+k\ell_{\alpha_0}}}X\bigr)$
	there exist $m\in\{1,\ldots,(\pi_1(X):A)\}$ and $[w]\in\pi_1(\mathcal{L}_{\alpha_0}X)$ such that
		\[
			m\langle\overline{u},[v]\rangle=(q_{\alpha_0}+k\ell_{\alpha_0})\langle\overline{u},[w]\rangle.
		\]
\end{enumerate}
\end{lemma}

\begin{proof}
(i)$\Rightarrow$(ii): Suppose that $u$ is aspherical.
Fix $\alpha\in [S^1,X]$ and $k\in\mathbb{Z}_{\geq0}$.
Let $f\colon X\to K=K(\pi_1(X),1)$ be the classifying map.
Hence $f$ induces an isomorphism of fundamental groups.
Applying Proposition \ref{proposition:iff},
there exists $\Omega\in H^2(K;\mathbb{R})$ such that
\[
	u=f^{\ast}\Omega.
\]
For every $[v\colon\mathbb{T}^2\to X]\in\pi_1(\mathcal{L}_{\alpha^{q_{\alpha}+k\ell_{\alpha}}}X)$, we have
\[
	\langle\overline{u},[v]\rangle%
	=\langle f^{\ast}\overline{\Omega},[v]\rangle%
	=\langle\overline{\Omega},[f\circ v]\rangle.
\]
We note that $[f\circ v]\in\pi_1\bigl(\mathcal{L}_{f_{\ast}(\alpha^{q_{\alpha}+k\ell_{\alpha}})}K\bigr)=\pi_1\bigl(\mathcal{L}_{f_{\ast}(\alpha)^{q_{\alpha}+k\ell_{\alpha}}}K\bigr)$,
where $f_{\ast}\colon[S^1,X]\to[S^1,K]$ is the map induced by $f$.
Moreover, Proposition \ref{proposition:centralizer} implies that
\[
	\pi_1\bigl(\mathcal{L}_{f_{\ast}(\alpha)^{q_{\alpha}+k\ell_{\alpha}}}K\bigr)%
	\cong C_{\pi_1(K)}\left(f_{\ast}(\gamma_{\alpha})^{q_{\alpha}+k\ell_{\alpha}}\right),
\]
where $f_{\ast}(\gamma_{\alpha})\in\pi_1(K)$ is a representative of the conjugacy class $f_{\ast}(\alpha)\in [S^1,K]$.

Denote by $c\in C_{\pi_1(K)}\left(f_{\ast}(\gamma_{\alpha})^{q_{\alpha}+k\ell_{\alpha}}\right)$ the image of $[f\circ v]\in\pi_1\bigl(\mathcal{L}_{f_{\ast}(\alpha)^{q_{\alpha}+k\ell_{\alpha}}}K\bigr)$
under the above isomorphism.
Applying Lemma \ref{lemma:bg} for $c$, we conclude that
there exists $m\in\{1,\ldots,(\pi_1(K):A)\}$ such that
\[
	c^m\in C_{\pi_1(K)}\bigl(f_{\ast}(\gamma_{\alpha})\bigr)\cong\pi_1\bigl(\mathcal{L}_{f_{\ast}(\alpha)}K\bigr).
\]
It implies that there exists $[w_0]\in\pi_1\bigl(\mathcal{L}_{f_{\ast}(\alpha)}K\bigr)$ such that
\[
	m\langle\overline{\Omega},[f\circ v]\rangle=(q_{\alpha}+k\ell_{\alpha})\langle\overline{\Omega},[w_0]\rangle.
\]
{\center
\begin{tikzpicture}
\draw (0,0) rectangle (1,1);
\draw (0,1) rectangle (1,2);
\draw (0,3) rectangle (1,4);
\fill[blue,opacity=.2] (0,0) rectangle (1,4);
\draw (1,0) rectangle (2,1);
\draw (3,0) rectangle (4,1);
\draw (3,3) rectangle (4,4);
\fill[red,opacity=.2] (0,0) rectangle (4,1);
\node at (2.52,0.5) {$\cdots$};
\node at (2.02,3.5) {$\cdots$};
\node at (0.5,2.6) {$\vdots$};
\node at (3.5,2.1) {$\vdots$};
\node at (2,2) {$\cdot$};
\node[xshift=3.535534,yshift=3.535534] at (2,2) {$\cdot$};
\node[xshift=-3.535534,yshift=-3.535534] at (2,2) {$\cdot$};
\foreach \x in {0.5,1.5,3.5}
	\node[below] at (\x,0.05) {$c$};
\foreach \y in {0.5,1.5,3.5}
	\node[left] at (0.05,\y) {$f_{\ast}(\gamma_{\alpha})$};
\draw [decorate,decoration={brace,amplitude=7pt},xshift=-32pt,yshift=0pt]%
(0,0) -- (0,4) node [midway,xshift=-1.1cm] {$q_{\alpha}+k\ell_{\alpha}$};
\draw [decorate,decoration={brace,mirror,amplitude=7pt},xshift=0pt,yshift=-10pt]%
(0,0) -- (4,0) node [midway,yshift=-0.5cm] {$m$};
\node[blue,above] at (0.5,4) {$\langle\overline{\Omega},[f\circ v]\rangle$};
\node[red,right] at (4,0.5) {$\langle\overline{\Omega},[w_0]\rangle$};
\end{tikzpicture}

}

Let $[w]\in\pi_1(\mathcal{L}_{\alpha}X)$ such that $f_{\ast}[w]=[w_0]$.
Then we have
\begin{align*}
	m\langle\overline{u},[v]\rangle%
	&=m\langle\overline{\Omega},[f\circ v]\rangle%
	=(q_{\alpha}+k\ell_{\alpha})\langle\overline{\Omega},[w_0]\rangle%
	=(q_{\alpha}+k\ell_{\alpha})\langle\overline{\Omega},[f\circ w]\rangle\\
	&=(q_{\alpha}+k\ell_{\alpha})\langle f^{\ast}\overline{\Omega},[w]\rangle%
	=(q_{\alpha}+k\ell_{\alpha})\langle \overline{u},[w]\rangle.
\end{align*}
Thus (ii) holds.

(iii)$\Rightarrow$(i): Suppose that $u$ is not aspherical.
Then $\langle u,\pi_2(X)\rangle$ is a non-trivial finitely generated $\mathbb{Z}$-submodule of $\mathbb{R}$.
We fix $\alpha\in [S^1,X]$ and choose a loop $z_{\alpha}$ representing $\alpha$.

We denote by $\Omega_{z_{\alpha}(0)}X\subset\mathcal{L}X$ the space of loops with base point $z_{\alpha}(0)$.
We define a map $\iota_1\colon\Omega_{z_{\alpha}(0)}X\to\mathcal{L}_{\alpha}X$ by concatenating a loop $x\in\Omega_{z_{\alpha}(0)}X$ with $z_{\alpha}$.
Then $\iota_1$ induces the homomorphism
\[
	\iota_{1\ast}\colon\pi_2(X,z_{\alpha}(0))\to\pi_1(\mathcal{L}_{\alpha}X,z_{\alpha}),
\]
where we used the fact that $\pi_1\bigl(\Omega_{z_{\alpha}(0)}X,z_{\alpha}(0)\bigr)\cong\pi_2(X,z_{\alpha}(0))$.
Similarly, for all $n\in\mathbb{N}$,
we can define the homomorphisms
\[
	\iota_{n\ast}\colon\pi_2(X)\to\pi_1(\mathcal{L}_{\alpha^n}X).
\]
Choose $s\in\pi_2(X)$ such that $\langle u,s\rangle\neq 0$.
Then we have
\[
	\langle u,s\rangle%
	\in\langle u,\pi_2(X)\rangle%
	=\langle\overline{u},\iota_{n\ast}(\pi_2(X))\rangle
	\subset\langle\overline{u},\pi_1(\mathcal{L}_{\alpha^n}X)\rangle
\]
for any $n\in\mathbb{N}$.
Hence it is enough to show that
for every $m=1,\ldots,(\pi_1(X):A)$ and every $[w]\in\pi_1(\mathcal{L}_{\alpha}X)$ we have
\[
	m\langle u,s\rangle\neq(q_{\alpha}+k\ell_{\alpha})\langle\overline{u},[w]\rangle
\]
when $k$ is large.

We note that
\[
	\langle u,\pi_2(X)\rangle%
	\subset\mathbb{Q}\langle u,\pi_2(X)\rangle\cap\langle\overline{u},\pi_1(\mathcal{L}_{\alpha}X)\rangle%
	\subset\langle\overline{u},\pi_1(\mathcal{L}_{\alpha}X)\rangle\subset\mathbb{R}.
\]
If
$\langle\overline{u},[w]\rangle\in\langle\overline{u},\pi_1(\mathcal{L}_{\alpha}X)\rangle\setminus(\mathbb{Q}\langle u,\pi_2(X)\rangle\cap\langle\overline{u},\pi_1(\mathcal{L}_{\alpha}X)\rangle)$,
then we have
\[
	\langle\overline{u},[w]\rangle\neq\frac{m}{q_{\alpha}+k\ell_{\alpha}}\langle u,s\rangle
\]
for any $m=1,\ldots,(\pi_1(X):A)$ and any $k\in\mathbb{Z}_{\geq 0}$.
If $\langle\overline{u},[w]\rangle\in\mathbb{Q}\langle u,\pi_2(X)\rangle\cap\langle\overline{u},\pi_1(\mathcal{L}_{\alpha}X)\rangle$
and $k$ is so large that
\[
	q_{\alpha}+k\ell_{\alpha}>(\pi_1(X):A)(\mathbb{Q}\langle u,\pi_2(X)\rangle\cap\langle\overline{u},\pi_1(\mathcal{L}_{\alpha}X)\rangle:\langle u,\pi_2(X)\rangle),
\]
then
\[
	\langle u,s\rangle\neq\frac{q_{\alpha}+k\ell_{\alpha}}{m}\langle\overline{u},[w]\rangle
\]
for any $m=1,\ldots,(\pi_1(X):A)$.

Since (iii) immediately follows from (ii), Lemma \ref{lemma:key} is proved.
\end{proof}

\begin{remark}
In general, we have
\[
	\langle\overline{u},\pi_1(\mathcal{L}_{\alpha^n}X)\rangle\supset n\langle\overline{u},\pi_1(\mathcal{L}_{\alpha}X)\rangle
\]
for any $u\in H^2(X;\mathbb{R})$, $\alpha\in [S^1,X]$ and $n\in\mathbb{N}$.
\end{remark}

\subsubsection{$\mathrm{R}$-groups}

Here we consider $\mathrm{R}$-groups.

\begin{definition}[\cite{Ko,Ku}]\label{definition:R}
A group $G$ is called an \textit{$\mathrm{R}$-group} if the equality $g^n=h^n$ implies $g=h$,
where $g,h$ are any elements in $G$ and $n$ is any natural number.
\end{definition}

Let $G$ be an $\mathrm{R}$-group.
Then we have

\begin{proposition}\label{proposition:R}
Let $g\in G$ and $n\in\mathbb{N}$.
If $c\in C_G(g^n)$, then $c\in C_G(g)$.
\end{proposition}

\begin{proof}
Let $c\in C_G(g^n)$.
Then the equality $(cgc^{-1})^n=cg^nc^{-1}=g^n$ implies $cgc^{-1}=g$.
Hence $c\in C_G(g)$.
\end{proof}

Combining with the proof of Lemma \ref{lemma:key}, we then obtain

\begin{lemma}\label{lemma:keynilp}
Let $X$ be a finite CW-complex whose fundamental group is an $\mathrm{R}$-group and $u\in H^2(X;\mathbb{R})$.
Then the following conditions are equivalent.

\begin{enumerate}
	\item $u$ is aspherical,
	\item for every $\alpha\in [S^1,X]$ and $n\in\mathbb{N}$, we have
		\[
			\langle\overline{u},\pi_1(\mathcal{L}_{\alpha^n}X)\rangle=%
			n\langle \overline{u},\pi_1(\mathcal{L}_{\alpha}X)\rangle,
		\]
	\item there exists $\alpha_0\in [S^1,X]$ such that
		for every $n\in\mathbb{N}$, we have
		\[
			\langle\overline{u},\pi_1(\mathcal{L}_{\alpha_0^n}X)\rangle=%
			n\langle \overline{u},\pi_1(\mathcal{L}_{\alpha_0}X)\rangle.
		\]
\end{enumerate}
\end{lemma}


\section{Proof of Theorem \ref{theorem:main}}\label{section:main}

In this section, we state a refined version (Theorem \ref{theorem:main3}) of Theorem \ref{theorem:main} and prove the theorems.
Let $(M,\omega)$ be a closed symplectic manifold.
We recall that an isolated periodic orbit $x$ of $H$ is said to be \textit{homologically non-trivial}
if for some lift $\bar{x}\in\overline{\mathcal{L}_{\alpha}M}$ of $x$, the local Floer homology $\mathrm{HF}^{\mathrm{loc}}(H,\bar{x})$ of $H$ at $\bar{x}$ is non-zero (see \cite {GG10} for details).
Every non-degenerate fixed point $x$ is homologically non-trivial
since we have
\[
	\mathrm{HF}_{\ast}^{\mathrm{loc}}(H,\bar{x})\cong%
	\begin{cases}%
		\mathbb{Z}/2\mathbb{Z} & \text{if\: $\ast=\mu_{\mathrm{CZ}}(H,\bar{x})$},\\
		0 & \text{otherwise},
	\end{cases}
\]
where $\mu_{\mathrm{CZ}}(H,\bar{x})$ is the Conley--Zehnder index of $\bar{x}$.
Then we can refine Theorem \ref{theorem:main} as follows (see also \cite[Theorem 3.1]{Gu13}).

\begin{theorem}\label{theorem:main3}
Assume that $\omega$ is aspherical and $\pi_1(M)$ is either a virtually abelian group or an $\mathrm{R}$-group.
Let $H\colon S^1\times M\to\mathbb{R}$ be a Hamiltonian
having an isolated and homologically non-trivial one-periodic orbit $x$ in the homotopy class $\alpha$
such that $[\alpha]\neq 0$ in $H_1(M;\mathbb{Z})/\mathrm{Tor}$, $\mathcal{P}_1(H;[\alpha])$ is finite
and $\omega$ is $\alpha$-toroidally rational.
Then for every sufficiently large prime $p_i\in P_{q_{\alpha},\ell_{\alpha}}$, the Hamiltonian $H$ has a simple periodic orbit in the homotopy class $\alpha^{p_i}$
and with period either $p_i$ or $p_{i+1}$.
Moreover, when $\pi_1(M)$ is an $\mathrm{R}$-group, then the finiteness condition on $\mathcal{P}_1(H;[\alpha])$
can be replaced by that on $\mathcal{P}_1(H;\alpha)$.
\end{theorem}

Here, when $\pi_1(M)$ is virtually abelian, we choose an abelian subgroup $A<\pi_1(M)$ of finite index,
$\gamma_{\alpha}\in\pi_1(M)$ representing $\alpha$,
and $\ell_{\alpha}\in\{1,\ldots,(\pi_1(M):A)\}$ such that $\gamma_{\alpha}^{\ell_{\alpha}}\in A$.
When $\pi_1(M)$ is an $\mathrm{R}$-group, we may choose an arbitrary positive integer $\ell_{\alpha}$.
As above, $q_{\alpha}$ is an arbitrary positive integer coprime to $\ell_{\alpha}$.
The proof of Theorem \ref{theorem:main3} is inspired by the argument by G\"urel \cite{Gu13}.

\begin{proof}[Proof: the virtually abelian case]
Since $\mathcal{P}_1(H;[\alpha])$ is finite, there exist finitely many distinct homotopy classes $\alpha_j\in [S^1,M]$
representing $[\alpha]\in H_1(M;\mathbb{Z})/\mathrm{Tor}$ such that every $x\in\mathcal{P}_1(H;[\alpha])$ is contained in one of $\alpha_j$'s.
As in \cite{GG16}, one can show that for every sufficiently large prime $p$, the classes $\alpha_j^p$ are all distinct
(If we replace the finiteness condition on $\mathcal{P}_1(H;[\alpha])$ with that on $\mathcal{P}_1(H;\alpha)$,
then there might exist $\beta\neq\alpha$ such that $\beta^p=\alpha^p$ even when $p$ is large.
However, if $\pi_1(M)$ is an $\mathrm{R}$-group, then $\gamma_{\alpha}^p$ has the unique $p$-th root $\gamma_{\alpha}$ and hence the conjugacy class $\alpha^p$
has the unique $p$-th root $\alpha$).

Fix a reference loop $z_{\alpha}\in\alpha$ and
choose the iterated loop $z_{\alpha}^p$ as the reference loop for $\alpha^p$.
Denote by $x_k$ the elements of $\mathcal{P}_1(H;\alpha)$.
We note that every sufficiently large prime $p$ is \textit{admissible} in the sense of \cite{GG10} for all orbits $x_k$
(i.e., $\lambda^p\neq 1$ for all eigenvalues $\lambda\neq 1$ of $(d\varphi_H)_{x_k}\colon T_{x_k}M\to T_{x_k}M$).
Since $x$ is isolated and homologically non-trivial, we have
$\mathrm{HF}_{\ast}^{\mathrm{loc}}(H,\bar{x})\neq 0$
for some lift $\bar{x}=[x,\Pi]\in\overline{\mathcal{L}_{\alpha}M}$ of $x$ and some $\ast\in\mathbb{Z}$.
By \cite[Theorem 1.1 and Remark 1.1]{GG10}, when $p$ is admissible,
we can think of $x^p$ as an isolated one-periodic orbit of $H^{\natural p}$ and
we have
\[
	\mathrm{HF}_{\ast+s_p}^{\mathrm{loc}}(H^{\natural p},\bar{x}^p)\cong \mathrm{HF}_{\ast}^{\mathrm{loc}}(H,\bar{x})
\]
for some $s_p$, where $\bar{x}^p=[x^p,\Pi^p]\in\overline{\mathcal{L}_{\alpha^p}M}$.
Hence we have $\mathrm{HF}_{\ast+s_p}^{\mathrm{loc}}(H^{\natural p},\bar{x}^p)\neq 0$.

From now on, we only consider primes in $P_{q_{\alpha},\ell_{\alpha}}$ (see \eqref{eq:pl} in Section \ref{section2} for the definition).
Let $p_i\in P_{q_{\alpha},\ell_{\alpha}}$ be a sufficiently large prime satisfying the above conditions.
Assume that $H$ has no simple $p_i$-periodic orbit in $\alpha^{p_i}$.
Since $p_i$ is prime, all $p_i$-periodic orbits in $\alpha^{p_i}$
are the $p_i$-th iterations of one-periodic orbits in $\alpha$.
Hence there is an action-preserving one-to-one correspondence between
$\mathcal{P}_1(H^{\natural p_i};\alpha^{p_i})$
and the set of $p_i$-th iterations $\left\{\,y^{p_i}\relmiddle| y\in\mathcal{P}_1(H;\alpha)\,\right\}$.

By adding a constant to the Hamiltonian $H$, we can assume that the action of the lift $\bar{x}$ is
$\mathcal{A}_H(\bar{x})=0$.
Hence for all $n\in\mathbb{N}$, we have
\[
	\mathcal{A}_{H^{\natural n}}(\bar{x}^n)=n\mathcal{A}_H(\bar{x})=0.
\]
Since $\mathcal{P}_1(H;[\alpha])$ is finite and $\omega$ is $\alpha$-toroidally rational,
we can choose $c>0$ so small that for all $m=1,\ldots,(\pi_1(M):A)$
\begin{equation}\label{eq:fracN}
	[-c,c)\cap\left\{\,\mathcal{A}_H(\bar{y})-\frac{1}{m}\langle\overline{[\omega]},[w]\rangle\relmiddle|%
	\bar{y}\in\overline{\mathcal{P}}_1(H;\alpha),\ [w]\in\pi_1(\mathcal{L}_{\alpha}M)\,\right\}=\{0\}.
\end{equation}
In particular, when $m=1$, we have
\[
	[-c,c)\cap\mathrm{Spec}(H;\alpha)=\{0\}.
\]

Now we claim that
\[
	[-p_ic,p_ic)\cap\mathrm{Spec}(H^{\natural p_i};\alpha^{p_i})=\{0\}.
\]
To see this, choose $s\in[-p_ic,p_ic)\cap\mathrm{Spec}(H^{\natural p_i};\alpha^{p_i})$.
Then there exist $\bar{y}\in\overline{\mathcal{P}}_1(H;\alpha)$ and $[v]\in\pi_1(\mathcal{L}_{\alpha^{p_i}}M)$ such that
\[
	s=\mathcal{A}_{H^{\natural p_i}}(\bar{y}^{p_i}\#[v])=p_i\mathcal{A}_H(\bar{y})-\langle \overline{[\omega]},[v]\rangle.
\]
Since $\pi_1(M)$ is virtually abelian and $\omega$ is aspherical, by applying Lemma \ref{lemma:key} for $[v]$,
there exist $m\in\{1,\ldots,(\pi_1(M):A)\}$ and $[w]\in\pi_1(\mathcal{L}_{\alpha}M)$ such that
\[
	m\langle\overline{[\omega]},[v]\rangle=p_i\langle\overline{[\omega]},[w]\rangle.
\]
Therefore,
\[
	\lvert s\rvert=p_i\left\lvert\mathcal{A}_H(\bar{y})-\frac{1}{m}\langle\overline{[\omega]},[w]\rangle\right\rvert<p_ic.
\]
By \eqref{eq:fracN}, it concludes that
\[
	\mathcal{A}_H(\bar{y})-\frac{1}{m}\langle \overline{[\omega]},[w]\rangle=0.
\]
Thus we obtain $s=0$.

Hence zero is the only critical value of $\mathcal{A}_{H^{\natural p_i}}$ in $[-p_ic,p_ic)$.
Therefore,
\[
	\mathrm{HFN}_{\ast}^{[-p_ic,p_ic)}(H^{\natural p_i};\alpha^{p_i})\cong\mathrm{HF}_{\ast}^{\mathrm{loc}}(H^{\natural p_i},\bar{x}^{p_i})\oplus\cdots,
\]
where the dots represent the contributions of the local Floer homology groups of $\bar{x}_k^{p_i}$
whose $\mathcal{A}_{H^{\natural p_i}}(\bar{x}_k^{p_i})=0$ and $\mu_{\mathrm{CZ}}(H^{\natural p_i},\bar{x}_k^{p_i})=\ast+s_{p_i}$
for some lifts $\bar{x}_k\in\overline{\mathcal{P}}_1(H;\alpha)$.
For any $[v]\in\Gamma_{\alpha^{p_i}}$, we have
\[
	0\neq\mathrm{HF}_{\ast+s_{p_i}}^{\mathrm{loc}}(H^{\natural p_i},\bar{x}^{p_i})\cong\mathrm{HF}_{\ast+s_{p_i}-2\langle\overline{c_1},[v]\rangle}^{\mathrm{loc}}(H^{\natural p_i},\bar{x}^{p_i}\#[v]).
\]
Hence
\[
	\mathrm{HFN}_{\ast}^{[-p_ic,p_ic)}(H^{\natural p_i};\alpha^{p_i})\cong\mathrm{HF}_{\ast}^{\mathrm{loc}}(H^{\natural p_i},\bar{x}^{p_i})\oplus\cdots\neq 0.
\]

We set
\[
	C=\max\left\{\int_{S^1}\max_M H_t\,dt,0\right\}+\max\left\{-\int_{S^1}\min_M H_t\,dt,0\right\}.
\]
Since $p_{i+1}-p_i=o(p_i)$ as $i\to\infty$ (see, e.g., \cite[Theorem 3 (I)]{BHP}),
we may assume $p_i\in P_{q_{\alpha},\ell_{\alpha}}$ so large that $p_ic>6C(p_{i+1}-p_i)$.
Choose $K>0$ such that
\[
	p_ic-4C(p_{i+1}-p_i)<K<p_ic-2C(p_{i+1}-p_i).
\]
Then we have
\begin{equation}\label{eq:inequality}
	-p_ic<-K<-K+2C(p_{i+1}-p_i)<0<K<K+2C(p_{i+1}-p_i)<p_ic,
\end{equation}
and
\[
	-p_{i+1}c<-K+C(p_{i+1}-p_i)<0<K+C(p_{i+1}-p_i)<p_{i+1}c.
\]
We set $\delta=C(p_{i+1}-p_i)$.
Now we have the following commutative diagram:
\[
	\xymatrix{
	\mathrm{HFN}^{[-K,K)}(H^{\natural p_i};\alpha^{p_i}) \ar[d]_{\sigma_{H^{\natural p_{i+1}}H^{\natural p_i}}} \ar[rrd]^{\cong} & & \\
	\mathrm{HFN}^{[-K+\delta,K+\delta)}(H^{\natural p_{i+1}};\alpha^{p_i}) \ar[rr]_{\sigma_{H^{\natural p_i}H^{\natural p_{i+1}}}} & & \mathrm{HFN}^{[-K+2\delta,K+2\delta)}(H^{\natural p_i};\alpha^{p_i})
	}
\]
Here the map $\sigma_{H^{\natural p_{i+1}}H^{\natural p_i}}$ (resp.\ $\sigma_{H^{\natural p_i}H^{\natural p_{i+1}}}$) is induced by
a linear homotopy from $H^{\natural p_i}$ to $H^{\natural p_{i+1}}$ (resp.\ from $H^{\natural p_{i+1}}$ to $H^{\natural p_i}$),
and the diagonal map is an isomorphism induced by the natural quotient-inclusion map (see \eqref{eq:inequality}).
Combining with $\mathrm{HFN}^{[-K,K)}(H^{\natural p_i};\alpha^{p_i})\neq 0$, we conclude that
\[
	\mathrm{HFN}^{[-K+\delta,K+\delta)}(H^{\natural p_{i+1}};\alpha^{p_i})\neq 0.
\]
Thus $H$ has a $p_{i+1}$-periodic orbit $y$ in the homotopy class $\alpha^{p_i}$, and hence in the homology class $p_i[\alpha]$.

Now it is enough to show that $y$ is simple.
Arguing by contradiction, we assume that $y$ is not simple.
Since $p_{i+1}$ is prime, $y$ is the $p_{i+1}$-th iteration of a one-periodic orbit in the homology class $p_i[\alpha]/p_{i+1}\in H_1(M;\mathbb{Z})/\mathrm{Tor}$.
Since $p_i/p_{i+1}$ is not an integer, this contradicts the fact that $[\alpha]\neq 0\in H_1(M;\mathbb{Z})/\mathrm{Tor}$.
\end{proof}

\begin{proof}[Proof: the $\mathrm{R}$-group case]
The proof is almost same as in the virtually abelian case
except at the following place.

Assume that all $p$-periodic orbits in $\alpha^p$
are the $p$-th iterations of one-periodic orbits in $\alpha$ for a large prime $p$.
Since $\pi_1(M)$ is an $\mathrm{R}$-group and $\omega$ is aspherical,
Lemma \ref{lemma:keynilp} shows that
\begin{equation}\label{eq:spechomog}
	\mathrm{Spec}(H^{\natural p};\alpha^p)=p\mathrm{Spec}(H;\alpha).
\end{equation}
Since $\mathcal{P}_1(H;\alpha)$ is finite and $\omega$ is $\alpha$-toroidally rational,
we can choose $c>0$ so small that
\[
	[-c,c)\cap\mathrm{Spec}(H;\alpha)=\{0\}.
\]
By \eqref{eq:spechomog}, we obtain
\[
	[-pc,pc)\cap\mathrm{Spec}(H^{\natural p};\alpha^p)=\{0\}.
\]
Then the rest of the proof follows the same path as in the virtually abelian case.
\end{proof}


\section{Proof of Theorems \ref{theorem:main2} and \ref{theorem:main2nilp}}

In this section, we show Theorems \ref{theorem:main2} and \ref{theorem:main2nilp}.
The main tool used here is the augmented action filtration on the Floer--Novikov homology.

\subsection{Augmented action filtered Floer--Novikov homology}\label{subsection:augmentedaction}

This subsection is devoted to introduce the augmented action filtration (see \cite{GG09,GG16}).
Let $(M,\omega)$ be a connected closed monotone or negative monotone symplectic manifold of dimension $2n$ with monotonicity constant $\lambda$.
Let $H\colon S^1\times M\to\mathbb{R}$ be a Hamiltonian.
For a free homotopy class $\alpha\in [S^1,M]$,
we fix a reference loop $z_{\alpha}\in\alpha$ and choose a trivialization of $TM$ along $z_{\alpha}$.

\subsubsection{Augmented action}

We define the \textit{augmented action} of a capped one-periodic orbit $\bar{x}\in\overline{\mathcal{P}}_1(H;\alpha)$ to be
\[
	\widetilde{\mathcal{A}}_H(\bar{x})=\mathcal{A}_H(\bar{x})-\frac{\lambda}{2}\Delta_H(\bar{x}).
\]
This is introduced by \cite{GG09} for contractible orbits,
and by \cite{GG16} for non-contractible ones.

We assume that all iterated homotopy classes $\alpha^k$, $k\in\mathbb{N}$, are distinct and non-trivial.
We choose the iterated loop $z_{\alpha}^k$ with the iterated trivialization as the reference loop for $\alpha^k$.
As the usual action functional, the augmented action $\widetilde{\mathcal{A}}_H$ is also homogeneous with respect to iterations.
Namely,
\[
	\widetilde{\mathcal{A}}_{H^{\natural k}}(\bar{x}^k)=k\widetilde{\mathcal{A}}_H(\bar{x}).
\]
Moreover, for any $\bar{x}\in\overline{\mathcal{P}}_1(H;\alpha)$ and $[v]\in\pi_1(\mathcal{L}_{\alpha}M,z_{\alpha})$,
we have
\[
	\widetilde{\mathcal{A}}_H(\bar{x}\#[v])=\widetilde{\mathcal{A}}_H(\bar{x})-\langle \overline{[\omega]}-\lambda\overline{c_1},[v]\rangle.
\]
The \textit{augmented action spectrum} $\widetilde{\mathrm{Spec}}(H;\alpha)$ is defined to be
the set of values of the augmented action of capped one-periodic orbits in $\alpha$, i.e.,
\[
	\widetilde{\mathrm{Spec}}(H;\alpha)=\widetilde{\mathcal{A}}_H\bigl(\overline{\mathcal{P}}_1(H;\alpha)\bigr).
\]

Now we assume that $\omega$ is $\alpha$-toroidally rational, i.e.,
$\langle\overline{[\omega]},\pi_1(\mathcal{L}_{\alpha}M,z_{\alpha})\rangle=h_{\alpha}\mathbb{Z}$
for some non-negative real number $h_{\alpha}$.
Since $(M,\omega)$ is monotone or negative monotone, we have
\[
	\langle [\omega],\pi_2(M)\rangle=\lambda\langle c_1,\pi_2(M)\rangle=\lambda c_1^{\mathrm{min}}\mathbb{Z}.
\]
Hence $h_{\alpha}$ divides $\lambda c_1^{\mathrm{min}}$.
We put $\nu=\lambda c_1^{\mathrm{min}}/h_{\alpha}\in\mathbb{N}$ and $\xi=c_1^{\mathrm{min}}/c_{1,\alpha}^{\mathrm{min}}\in\mathbb{N}$.
We fix $[v_0]\in\Ker\overline{c_1}$ and $[w_0]\in\Ker\overline{[\omega]}$
and choose $n_v,n_w\in\mathbb{Z}$ such that $\langle\overline{[\omega]},[v_0]\rangle=h_{\alpha}n_v$
and $\langle\overline{c_1},[w_0]\rangle=c_{1,\alpha}^{\mathrm{min}}n_w$.
Then we obtain
\[
	\langle\overline{[\omega]}-\lambda\overline{c_1},[v_0]^{n_w\nu}\# [w_0]^{n_v\xi}\rangle%
	=n_w\nu h_{\alpha}n_v-\lambda n_v\xi c_{1,\alpha}^{\mathrm{min}}n_w%
	=0.
\]
We set an equivalence relation $\sim$ on $\overline{\mathcal{P}}_1(H;\alpha)$ by defining $[x_1,\Pi_1]\sim [x_2,\Pi_2]$ if and only if
$x_1=x_2$ and $[\Pi_1\#(-\Pi_2)]\in\{\,([v_0]^{n_w\nu}\# [w_0]^{n_v\xi})^k\mid k\in\mathbb{Z}_{\geq 0}\,\}$.
We denote by $\overline{\mathcal{P}}_1'(H;\alpha)$ the set of such equivalence classes $\bar{x}'=[x,\Pi]'$.
For $[v_0^k]\in\Ker\overline{c_1}$ and $[w_0^k]\in\Ker\overline{[\omega]}$, $k\in\mathbb{N}$, we define the set $\overline{\mathcal{P}}_1'(H;\alpha^k)$ in the same manner.

Let $I=[a,b)$ be an interval with $a,b\in\mathbb{R}\setminus\widetilde{\mathrm{Spec}}(H;\alpha)$.
We suppose that $H$ is $\alpha$-regular (i.e., all one-periodic orbits of $H$ representing $\alpha$ are non-degenerate).
Since $[\omega]-\lambda c_1$ is also $\alpha$-toroidally rational,
the number of $\bar{x}'\in\overline{\mathcal{P}}_1'(H;\alpha)$ with augmented action in $I$ is finite.
We define $\chi(H,I;\alpha)$ to be the sum of the Poincar\'e--Hopf indices of their Poincar\'e return maps.
Namely,
\[
	\chi(H,I;\alpha)=\sum_{\bar{x}'\in\overline{\mathcal{P}}_1'(H;\alpha),\: \widetilde{\mathcal{A}}_H(\bar{x}')\in I}%
	\sgn\det\bigl((d\varphi_H)_{x(0)}-\mathrm{id}\bigr).
\]
Since $\widetilde{\mathrm{Spec}}(K;\alpha)$ depends continuously on the Hamiltonian $K$
in the sense that for any open subsets $U,V\subset\mathbb{R}$ satisfying $V\subset U$
and for any Hamiltonian $H$ sufficiently $C^1$-close to $K$ we have $\widetilde{\mathrm{Spec}}(H;\alpha)\cap U\subset V$,
and $\chi(H,I;\alpha)$ takes values in $\mathbb{Z}$,
this definition can be extended to all Hamiltonians $K$ satisfying
$a,b\in\mathbb{R}\setminus\widetilde{\mathrm{Spec}}(K;\alpha)$.


\subsubsection{Augmented action filtration}

Here we give necessary changes in the argument of \cite[Subsection 3.3]{GG16} to be applicable to our case.
We define the \textit{augmented action gap} by
\[
	\gap(H;\alpha)=\inf\left\{\,\lvert s-s'\rvert\in [0,\infty]\relmiddle|\,s\neq s'\in\widetilde{\mathrm{Spec}}(H;\alpha)\right\}.
\]
We use the convention that $\inf\emptyset=\infty$.
We set
\[
	c_0(M)=\lvert\lambda\rvert\frac{2n\pm 1}{2},
\]
where $\pm$ is the sign of $\lambda$.
We say that the \textit{gap condition} is satisfied if
\[
	\gap(H;\alpha)>c_0(M).
\]
\begin{proposition}[{\cite[Proposition 3.1]{GG16}}]\label{proposition:augmented}
Assume that $H$ is $\alpha$-regular and
the gap condition is satisfied.
Then the complex $\mathrm{CFN}(H;\alpha)$, and hence $\mathrm{HFN}(H;\alpha)$, is filtered
by the augmented action.
In other words,
\[
	\widetilde{\mathcal{A}}_H(\bar{y}')\leq\widetilde{\mathcal{A}}_H(\bar{x}')
\]
whenever $\bar{y}'$ appears in $\partial\bar{x}'$ with non-zero coefficient.
\end{proposition}

Let $a$ and $b$ be real numbers such that $-\infty\leq a<b\leq\infty$ and $a,b\not\in\widetilde{\mathrm{Spec}}(H;\alpha)$.
We assume that $H$ is $\alpha$-regular
and the gap condition is satisfied.
We set $\widetilde{\mathcal{P}}_1^a=\{\,\bar{x}'\in\overline{\mathcal{P}}_1'(H;\alpha)\mid\widetilde{\mathcal{A}}_H(\bar{x}')<a\,\}$.
We define the augmented action filtered chain group by
\[
	\widetilde{\mathrm{CFN}}^{[a,b)}(H;\alpha)=\widetilde{\mathrm{CFN}}^b(H;\alpha)/\widetilde{\mathrm{CFN}}^a(H;\alpha),
\]
where
\[
	\widetilde{\mathrm{CFN}}^a(H;\alpha)=\left\{\,\xi=\sum\xi_{\bar{x}'}\bar{x}'\relmiddle|%
	\begin{aligned}%
		\bar{x}'\in\widetilde{\mathcal{P}}_1^a,\, \xi_{\bar{x}'}\in\mathbb{Z}/2\mathbb{Z}\ \text{such that $\forall C\in\mathbb{R}$,}\\%
		\#\{\,\bar{x}'\mid\xi_{\bar{x}'}\neq 0,\, \mathcal{A}_H(\bar{x}')>C\,\}<\infty%
	\end{aligned}\,\right\}.
\]
Proposition \ref{proposition:augmented} shows that $\widetilde{\mathrm{CFN}}^a(H;\alpha)$ is invariant under the boundary operator $\partial_b^{H,J}$.
Thus we get an induced operator $\partial_{[a,b)}^{H,J}$ on the quotient $\widetilde{\mathrm{CFN}}^{[a,b)}(H;\alpha)$.
Then the \textit{augmented action filtered Floer--Novikov homology group} is defined to be
\[
	\widetilde{\mathrm{HFN}}^{[a,b)}(H;\alpha)=\Ker{\partial_{[a,b)}^{H,J}}/\Image{\partial_{[a,b)}^{H,J}}.
\]

The following proposition enables us
to define the augmented action filtered Floer--Novikov homology $\widetilde{\mathrm{HFN}}^{[a,b)}(H;\alpha)$
even if $H$ is not $\alpha$-regular.

\begin{proposition}[{\cite[Proposition 3.3]{GG16}}]\label{proposition:augmented2}
Let $H\colon S^1\times M\to\mathbb{R}$ be a Hamiltonian
such that the gap condition is satisfied and let $a\not\in\widetilde{\mathrm{Spec}}(H;\alpha)$.
Then for any $\alpha$-regular Hamiltonian $K$ sufficiently $C^1$-close to $H$, the subspace
$\widetilde{\mathrm{CFN}}^a(K;\alpha)\subset\mathrm{CFN}(K;\alpha)$ is a subcomplex.
\end{proposition}

We define the set
\[
	\widetilde{\mathcal{H}}^{a,b}(M;\alpha)=\{\,H\colon S^1\times M\to\mathbb{R}\mid a,b\not\in\widetilde{\mathrm{Spec}}(H;\alpha)\,\}.
\]

\begin{definition}
For $H\in\widetilde{\mathcal{H}}^{a,b}(M;\alpha)$,
we define $\widetilde{\mathrm{HFN}}^{[a,b)}(H;\alpha)=\widetilde{\mathrm{HFN}}^{[a,b)}(K;\alpha)$,
where $K$ is any $\alpha$-regular Hamiltonian sufficiently $C^1$-close to $H$.
\end{definition}

A standard argument similar to Subsection \ref{subsection:floernovikov} shows that
this definition does not depend on the choice of $K$.

\begin{remark}\label{remark:chinonzero}
Let $I=[a,b)$ be an interval with $a,b\in\mathbb{R}\setminus\widetilde{\mathrm{Spec}}(H;\alpha)$.
We suppose that $\omega$ is $\alpha$-toroidally rational.
Then a straightforward computation shows that
\begin{align*}
	\chi(H,I;\alpha)%
	&=\sum_{\bar{x}'\in\overline{\mathcal{P}}_1'(H;\alpha),\ \widetilde{\mathcal{A}}_H(\bar{x}')\in I}%
	\sgn\det\bigl((d\varphi_H)_{x(0)}-\mathrm{id}\bigr)\\
	&=\sum_{\bar{x}'\in\overline{\mathcal{P}}_1'(H;\alpha),\ \widetilde{\mathcal{A}}_H(\bar{x}')\in I}%
	(-1)^{\mu_{\mathrm{CZ}}(H,\bar{x}')-n}\\
	&=(-1)^n\left\{\dim_{\mathbb{Z}/2\mathbb{Z}}\widetilde{\mathrm{CFN}}_{\mathrm{even}}^I(H;\alpha)%
	-\dim_{\mathbb{Z}/2\mathbb{Z}}\widetilde{\mathrm{CFN}}_{\mathrm{odd}}^I(H;\alpha)\right\}\\
	&=(-1)^n\left\{\dim_{\mathbb{Z}/2\mathbb{Z}}\widetilde{\mathrm{HFN}}_{\mathrm{even}}^I(H;\alpha)%
	-\dim_{\mathbb{Z}/2\mathbb{Z}}\widetilde{\mathrm{HFN}}_{\mathrm{odd}}^I(H;\alpha)\right\}.
\end{align*}
In particular, we have $\widetilde{\mathrm{HFN}}^I(H;\alpha)\neq 0$ if $\chi(H,I;\alpha)\neq 0$.
Here we note that if one of the conditions that $a\neq-\infty$, $b\neq\infty$ and $\omega$ is $\alpha$-toroidally rational is dropped,
then the $\mathbb{Z}/2\mathbb{Z}$-vector space $\widetilde{\mathrm{CFN}}^I(H;\alpha)$ might be infinite-dimensional.
\end{remark}


\subsubsection{Continuation}

Let $H^-, H^+\colon S^1\times M\to\mathbb{R}$ be two Hamiltonians.
We consider a linear homotopy $\{H_s\}_{s\in\mathbb{R}}$ from $H^-$ to $H^+$ (see Subsection \ref{subsection:floernovikov}).
We set
\[
	c_a(H_s)=\int_{S^1}\max_M(H^+-H^-)\,dt
\]
and
\[
	c_h(H_s)=\max\{0,c_a(H_s)\}+\lvert\lambda\rvert n\geq 0.
\]

\begin{proposition}[{\cite[Proposition 3.5]{GG16}}]\label{proposition:augmented3}
Assume that both $H^-$ and $H^+$ satisfy the gap condition, i.e.,
\[
	\gap(H^-;\alpha)>c_0(M)\quad \text{and}\quad \gap(H^+;\alpha)>c_0(M),
\]
and $a,b\in\mathbb{R}\cup\{\infty\}$ satisfy
$a<b$ and $a,b\not\in\widetilde{\mathrm{Spec}}(H^{\pm};\alpha)$.
Then a homotopy $\{H_s\}_s$ from $H^-$ to $H^+$ induces a map in the Floer--Novikov homology shifting the action
filtration upward by $c_h(H_s)$$:$
\[
	\sigma_{H^+H^-}\colon\widetilde{\mathrm{HFN}}^{[a,b)}(H^-;\alpha)\to\widetilde{\mathrm{HFN}}^{[a,b)+c_h(H_s)}(H^+;\alpha),
\]
where $[a,b)+c_h(H_s)=[a+c_h(H_s),b+c_h(H_s))$.
\end{proposition}


\subsection{Proof of Theorem \ref{theorem:main2}}

As in Theorem \ref{theorem:main3}, we choose an abelian subgroup $A<\pi_1(M)$ of finite index,
$\gamma_{\alpha}\in\pi_1(M)$ representing $\alpha$,
a positive integer $\ell_{\alpha}\in\{1,\ldots,(\pi_1(M):A)\}$ such that $\gamma_{\alpha}^{\ell_{\alpha}}\in A$
and a positive integer $q_{\alpha}$ coprime to $\ell_{\alpha}$.
The proof is inspired by the argument by Ginzburg and G\"urel \cite{GG16}.

\begin{proof}
Since $\mathcal{P}_1(H;[\alpha])$ is finite, there exist finitely many distinct homotopy classes $\alpha_j\in [S^1,M]$
representing $[\alpha]\in H_1(M;\mathbb{Z})/\mathrm{Tor}$ such that every $x\in\mathcal{P}_1(H;[\alpha])$ is contained in one of $\alpha_j$'s.
As in Theorem \ref{theorem:main3}, one can show that for every sufficiently large prime $p$, the classes $\alpha_j^p$ are all distinct.
Fix a reference loop $z_{\alpha}\in\alpha$ and a trivialization of $TM|_{z_{\alpha}}$.
Choose the iterated loop $z_{\alpha}^p$ with the iterated trivialization as the reference loop for $\alpha^p$.

From now on, we only consider primes in $P_{q_{\alpha},\ell_{\alpha}}$ (see \eqref{eq:pl} in Section \ref{section2} for the definition).
Let $p_i\in P_{q_{\alpha},\ell_{\alpha}}$ be a sufficiently large prime satisfying the above condition.
Assume that $H$ has no simple $p_i$-periodic orbit in $\alpha^{p_i}$.
Since $p_i$ is prime, all $p_i$-periodic orbits in $\alpha^{p_i}$ are the $p_i$-th iterations of one-periodic orbits in $\alpha$.
Hence there is an action-preserving and mean index-preserving one-to-one correspondence between
$\mathcal{P}_1(H^{\natural p_i};\alpha^{p_i})$
and the set of $p_i$-th iterations $\left\{\,y^{p_i}\relmiddle| y\in\mathcal{P}_1(H;\alpha)\,\right\}$.

Put
\[
	\mathcal{S}=\left\{\,\widetilde{\mathcal{A}}_H(\bar{x}')-\frac{1}{m}\langle\overline{[\omega]}-\lambda\overline{c_1},[w]\rangle\relmiddle|%
	\begin{aligned}%
		\bar{x}'\in\overline{\mathcal{P}}_1'(H;\alpha),\ [w]\in\pi_1(\mathcal{L}_{\alpha}M),\\%
		m=1,\ldots,(\pi_1(M):A)%
	\end{aligned}\,\right\}.
\]
Since $\chi(H,I;\alpha)\neq0$ for every sufficiently small interval $I$ centered at some $s\in\widetilde{\mathrm{Spec}}(H;\alpha)$,
we can assume that $\chi(H,[-c,c);\alpha)\neq0$ for every sufficiently small $c>0$ by adding a constant to the Hamiltonian $H$.
Moreover, since $\mathcal{P}_1(H;[\alpha])$ is finite and $\omega$ is $\alpha$-toroidally rational
(and so is $[\omega]-\lambda c_1$),
we can choose $c>0$ so small that
\[
	[-c,c]\cap\mathcal{S}=\{0\}.
\]
In particular, we have $[-c,c)\cap\widetilde{\mathrm{Spec}}(H;\alpha)=\{0\}$.
Since the monotonicity of $(M,\omega)$ implies the asphericity of $[\omega]-\lambda c_1$
and $\pi_1(M)$ is virtually abelian,
as in the proof of Theorem \ref{theorem:main3}, one can show that
\begin{equation}\label{eq:augspec}
	[-p_ic,p_ic)\cap\widetilde{\mathrm{Spec}}(H^{p_i};\alpha^{p_i})=\{0\}.
\end{equation}
Therefore, there is a one-to-one correspondence between the sets
\[
	\left\{\,\bar{x}'\in\overline{\mathcal{P}}_1'(H;\alpha)\relmiddle|\widetilde{\mathcal{A}}_H(\bar{x}')\in I\,\right\}%
	\ \text{and}\ %
	\left\{\,\bar{x}'\in\overline{\mathcal{P}}_1'(H^{\natural p_i};\alpha^{p_i})\relmiddle|\widetilde{\mathcal{A}}_{H^{\natural p_i}}(\bar{x}')\in p_iI\,\right\},
\]
where $I=[-c,c)$ and $p_iI=[-p_ic,p_ic)$.
Moreover, the Shub--Sullivan theorem \cite{SS,CMPY} shows that the Poincar\'e--Hopf index of $x^{p_i}$ coincides with that of $x$
for sufficiently large admissible (see the proof of Theorem \ref{theorem:main3} or \cite{GG10} for the definition) prime $p_i\in P_{q_{\alpha},\ell_{\alpha}}$.
Therefore, we have
\[
	\chi(H^{\natural p_i},p_iI;\alpha^{p_i})=\chi(H,I;\alpha)
\]
when $p_i\in P_{q_{\alpha},\ell_{\alpha}}$ is large.

Now we claim that we can define the augmented action filtered Floer--Novikov homology $\widetilde{\mathrm{HFN}}^{p_iI}(H^{\natural p_i};\alpha^{p_i})$
as long as $p_i\in P_{q_{\alpha},\ell_{\alpha}}$ is so large that $p_ic>c_0(M)$.
Indeed, let $\bar{x}'$ be a generator of $\mathrm{CFN}(H^{\natural p_i};\alpha^{p_i})$ with $\widetilde{\mathcal{A}}_{H^{\natural p_i}}(\bar{x}')\in p_iI$.
By the choice of $c$, we have $\widetilde{\mathcal{A}}_{H^{\natural p_i}}(\bar{x}')=0$.
Hence it is enough to show that $\widetilde{\mathcal{A}}_{H^{\natural p_i}}(\bar{y}')\leq 0=\widetilde{\mathcal{A}}_{H^{\natural p_i}}(\bar{x}')$
whenever $\bar{y}'$ appears in $\partial\bar{x}'$ with non-zero coefficient.
For simplicity, we assume that $\lambda>0$.
Then
\begin{align*}
	\widetilde{\mathcal{A}}_{H^{\natural p_i}}(\bar{y}')%
	&=\mathcal{A}_{H^{\natural p_i}}(\bar{y}')-\frac{\lambda}{2}\Delta_{H^{\natural p_i}}(\bar{y}')\\
	&\leq\mathcal{A}_{H^{\natural p_i}}(\bar{y}')-\frac{\lambda}{2}(\mu_{\mathrm{CZ}}(H^{\natural p_i},\bar{y}')-n)\\
	&<\mathcal{A}_{H^{\natural p_i}}(\bar{x}')-\frac{\lambda}{2}(\mu_{\mathrm{CZ}}(H^{\natural p_i},\bar{x}')-n-1)\\
	&\leq\mathcal{A}_{H^{\natural p_i}}(\bar{x}')-\frac{\lambda}{2}(\Delta_{H^{\natural p_i}}(\bar{x}')-2n-1)\\
	&=\widetilde{\mathcal{A}}_{H^{\natural p_i}}(\bar{x}')+c_0(M)\\
	&<p_ic.
\end{align*}
By \eqref{eq:augspec}, we have either $\widetilde{\mathcal{A}}_{H^{\natural p_i}}(\bar{y}')<-p_ic<0$ or $\widetilde{\mathcal{A}}_{H^{\natural p_i}}(\bar{y}')=0$.

Since $\chi(H^{\natural p_i},p_iI;\alpha^{p_i})=\chi(H,I;\alpha)\neq 0$, Remark \ref{remark:chinonzero} shows that
\[
	\widetilde{\mathrm{HFN}}^{p_iI}(H^{\natural p_i};\alpha^{p_i})\neq 0.
\]

Let $\{H_s^+\}_s$ (resp.\ $\{H_s^-\}_s$) be a linear homotopy from $H^{\natural p_i}$ to $H^{\natural p_{i+1}}$
(resp.\ from $H^{\natural p_{i+1}}$ to $H^{\natural p_i}$).
We set
\[
	e_+=\max\left\{\int_{S^1}\max_M H_t\,dt,0\right\},\quad e_-=\max\left\{-\int_{S^1}\min_M H_t\,dt,0\right\}
\]
and
\[
	a_{\pm}=(p_{i+1}-p_i)e_{\pm}+\lvert\lambda\rvert n.
\]
Then we have
\begin{equation}\label{eq:ac}
	a_{\pm}\geq c_h(H_s^{\pm}).
\end{equation}

Since $p_{i+1}-p_i=o(p_i)$ as $i\to\infty$ (see, e.g., \cite[Theorem 3 (I)]{BHP}), we have
\[
	\frac{1}{2}p_ic>a_++a_-
\]
for sufficiently large $p_i\in P_{q_{\alpha},\ell_{\alpha}}$.
Let $J=[-c/2,c/2)\subset I$.
Then
\[
	(p_iJ+a_++a_-)\cap\widetilde{\mathrm{Spec}}(H^{p_i};\alpha^{p_i})=\{0\}.
\]
Then $\widetilde{\mathrm{HFN}}^{p_iJ+a_++a_-}(H^{\natural p_i};\alpha^{p_i})$ is also defined when $p_i\in P_{q_{\alpha},\ell_{\alpha}}$ is so large that
$p_ic+a_++a_->c_0(M)$.
Moreover, we have an isomorphism
\[
	\Phi\colon\widetilde{\mathrm{HFN}}^{p_iJ}(H^{\natural p_i};\alpha^{p_i})\to\widetilde{\mathrm{HFN}}^{p_iJ+a_++a_-}(H^{\natural p_i};\alpha^{p_i})
\]
induced by the natural quotient-inclusion map.

Now we are in a position to show that $H$ has a $p_{i+1}$-periodic orbit in the homotopy class $\alpha^{p_i}$.
Arguing by contradiction, we assume that there are no such orbits.
Then we have $\gap(H^{\natural p_{i+1}};\alpha^{p_i})=\infty$.
By Proposition \ref{proposition:augmented}, the filtered Floer--Novikov homology $\widetilde{\mathrm{HFN}}^{p_iJ+a_+}(H^{\natural p_{i+1}};\alpha^{p_i})$ is defined
(of course, this should be zero at the chain level).
Once the filtered Floer--Novikov homology groups are defined,
it is easy to see that the same conclusion as with Proposition \ref{proposition:augmented3} holds.
Hence we have the continuation maps
\[
	\sigma_{H^{\natural p_{i+1}}H^{\natural p_i}}\colon\widetilde{\mathrm{HFN}}^{p_iJ}(H^{\natural p_i};\alpha^{p_i})\to\widetilde{\mathrm{HFN}}^{p_iJ+a_+}(H^{\natural p_{i+1}};\alpha^{p_i})
\]
and
\[
	\sigma_{H^{\natural p_i}H^{\natural p_{i+1}}}\colon\widetilde{\mathrm{HFN}}^{p_iJ+a_+}(H^{\natural p_{i+1}};\alpha^{p_i})\to\widetilde{\mathrm{HFN}}^{p_iJ+a_++a_-}(H^{\natural p_i};\alpha^{p_i})
\]
by \eqref{eq:ac}.
Now we have the following commutative diagram:
\[
	\xymatrix{
	\widetilde{\mathrm{HFN}}^{p_iJ}(H^{\natural p_i};\alpha^{p_i}) \ar[d]_{\sigma_{H^{\natural p_{i+1}}H^{\natural p_i}}} \ar[rrd]^(0.44){\Phi} & & \\
	\widetilde{\mathrm{HFN}}^{p_iJ+a_+}(H^{\natural p_{i+1}};\alpha^{p_i}) \ar[rr]_{\sigma_{H^{\natural p_i}H^{\natural p_{i+1}}}} & & \widetilde{\mathrm{HFN}}^{p_iJ+a_++a_-}(H^{\natural p_i};\alpha^{p_i})
	}
\]
Since $\Phi$ is an isomorphism and $\widetilde{\mathrm{HFN}}^{p_iJ}(H^{\natural p_i};\alpha^{p_i})\neq 0$, we conclude that
\[
	\widetilde{\mathrm{HFN}}^{p_iJ+a_+}(H^{\natural p_{i+1}};\alpha^{p_i})\neq 0.
\]
Thus $H$ has a $p_{i+1}$-periodic orbit $y$ in $\alpha^{p_i}$, and hence in the homology class $p_i[\alpha]$.

Now it is enough to show that $y$ is simple.
Arguing by contradiction, we assume that $y$ is not simple.
Since $p_{i+1}$ is prime, $y$ is the $p_{i+1}$-th iteration of a one-periodic orbit in the homology class $p_i[\alpha]/p_{i+1}\in H_1(M;\mathbb{Z})/\mathrm{Tor}$.
Since $p_i/p_{i+1}$ is not an integer, this contradicts the fact that $[\alpha]\neq 0\in H_1(M;\mathbb{Z})/\mathrm{Tor}$.
\end{proof}


\subsection{Proof of Theorem \ref{theorem:main2nilp}}

\begin{proof}
Since the proof is almost same as in Theorem \ref{theorem:main2},
here we only give the necessary changes.

Assume that all $p$-periodic orbits in $\alpha^p$
are the $p$-th iterations of one-periodic orbits in $\alpha$ for a large prime $p$.
Since $\pi_1(M)$ is an $\mathrm{R}$-group and $[\omega]-\lambda c_1$ is aspherical,
Lemma \ref{lemma:keynilp} shows that
\[
	\widetilde{\mathrm{Spec}}(H^{\natural p};\alpha^p)=p\widetilde{\mathrm{Spec}}(H;\alpha).
\]
Thus we obtain $\gap(H^{\natural p};\alpha^p)=p\gap(H;\alpha)$.
Hence, by Proposition \ref{proposition:augmented},
the augmented action filtered Floer--Novikov homology $\widetilde{\mathrm{HFN}}^{pI}(H^{\natural p};\alpha^p)$
is defined as long as $p$ is so large that $p\gap(H;\alpha)>c_0(M)$.

Moreover, since $\mathcal{P}_1(H;\alpha)$ is finite and $[\omega]-\lambda c_1$ is $\alpha$-toroidally rational,
we can show that
\[
	\chi(H^{\natural p},pI;\alpha^p)=\chi(H,I;\alpha)
\]
when $p$ is large.
Then the rest of the proof follows the same path as in Theorem \ref{theorem:main2}.
\end{proof}


\section*{Acknowledgments}

The author would like to express his sincere gratitude to his advisor Takashi Tsuboi and Urs Frauenfelder for many fruitful discussions.
The author is also grateful to Viktor Ginzburg, Ba\c{s}ak G\"urel, Morimichi Kawasaki, Yoshihiko Mitsumatsu and Shun Wakatsuki for many valuable comments.
A part of this work was carried out while the author was visiting the University of Augsburg.
The author would like to thank the institute for its warm hospitality and support.


\bibliographystyle{amsart}

\begin{thebibliography}{99}
\bibitem[Ba]{Ba} M. Bator\'eo,
	{\it On non-contractible periodic orbits of symplectomorphisms},
	J. Symplectic Geom.\ \textbf{15} (2017), no.~3, 687--717.
\bibitem[BPS]{BPS} P. Biran, L. Polterovich and D. Salamon,
	{\it Propagation in Hamiltonian dynamics and relative symplectic homology},
	Duke Math.\ J. \textbf{119} (2003), no.~1, 65--118.
\bibitem[BHP]{BHP} R. Baker, G. Harman and J. Pintz,
	{\it The exceptional set for Goldbach's problem in short intervals},
	Sieve methods, exponential sums, and their applications in number theory (Cardiff, 1995), 1--54, London Math.\ Soc.\ Lecture Note Ser., \textbf{237} (1997), Cambridge Univ.\ Press, Cambridge.
\bibitem[CMPY]{CMPY} S. Chow, J. Mallet-Paret and J. Yorke,
	{\it A periodic orbit index which is a bifurcation invariant},
	Geometric dynamics (Rio de Janeiro, 1981), 109--131, Lecture Notes in Math., 1007, Springer, Berlin (1983).
\bibitem[Co]{Co} C. Conley,
	{\it Lecture at the University of Wisconsin},
	April 6 (1984).
\bibitem[Di]{Di} P. Dirichlet,
	{\it Beweis des Satzes, dass jede unbegrenzte arithmetische Progression, deren erstes Glied und Differenz ganze Zahlen ohne gemeinschaftlichen Factor sind, unendlich viele Primzahlen enth\"alt},
	Abhandlungen der K\"oniglichen Preu\ss ischen Akademie der Wissenschaften zu Berlin \textbf{48} (1837), 45--71.
\bibitem[Fl]{Fl} A. Floer,
	{\it Symplectic fixed points and holomorphic spheres},
	Comm.\ Math.\ Phys.\ \textbf{120} (1989), 575--611.
\bibitem[FHS]{FHS} A. Floer, H. Hofer and D. Salamon,
	{\it Transversality in elliptic Morse theory for the symplectic action},
	Duke Math.\ J. \textbf{80} (1995), 251--292.
\bibitem[Fr92]{Fr92} J. Franks,
	{\it Geodesics on $S^2$ and periodic points of annulus homeomorphisms},
	Invent.\ Math.\ \textbf{108} (1992), no.~2, 403--418.
\bibitem[Fr96]{Fr96} J. Franks,
	{\it Area preserving homeomorphisms of open surfaces of genus zero},
	New York J.\ Math.\ \textbf{2} (1996), 1--19.
\bibitem[FO]{FO} K. Fukaya and K. Ono,
	{\it Arnold conjecture and Gromov--Witten invariant},
	Topology \textbf{38} (1999), 933--1048.
\bibitem[GL]{GL} D. Gatien and F. Lalonde,
	{\it Holomorphic cylinders with Lagrangian boundaries and Hamiltonian dynamics},
	Duke Math.\ J. \textbf{102} (2000), no.~3, 485--511.
\bibitem[GG09]{GG09} V. Ginzburg and B. G\"urel,
	{\it Action and index spectra and periodic orbits in Hamiltonian dynamics},
	Geom.\ Topol.\ \textbf{13} (2009), no.~5, 2745--2805.
\bibitem[GG10]{GG10} V. Ginzburg and B. G\"urel,
	{\it Local Floer homology and the action gap},
	J. Symplectic Geom.\ \textbf{8} (2010), no.~3, 323--357.
\bibitem[GG16]{GG16} V. Ginzburg and B. G\"urel,
	{\it Non-contractible periodic orbits in Hamiltonian dynamics on closed symplectic manifolds},
	Compos.\ Math.\ \textbf{152} (2016), no.~9, 1777--1799.
\bibitem[GG17]{GG17} V. Ginzburg and B. G\"urel,
	{\it Conley Conjecture Revisited},
	Int.\ Math.\ Res.\ Not.\ IMRN, rnx137, \texttt{https://doi.org/10.1093/imrn/rnx137}.
\bibitem[G\"u13]{Gu13} B. G\"urel,
	{\it On non-contractible periodic orbits of Hamiltonian diffeomorphisms},
	Bull.\ Lond.\ Math.\ Soc.\ \textbf{45} (2013), no.~6, 1227--1234.
\bibitem[Ha]{Ha} V. Hansen,
	{\it On the fundamental group of a mapping space. An example},
	Compos.\ Math.\ \textbf{28} (1974), no.~1, 33--36.
\bibitem[Hi]{Hi} N. Hingston,
	{\it Subharmonic solutions of Hamiltonian equations on tori},
	Ann.\ of Math.\ (2) \textbf{170} (2009), no.~2, 529--560.
\bibitem[HS]{HS} H. Hofer, D. Salamon,
	{\it Floer homology and Novikov rings},
	in The Floer Memorial Volume, 483--524, Progr.\ Math., 133, Birkhuser, Basel, 1995.
\bibitem[HZ]{HZ} H. Hofer and E. Zehnder,
	{\it Symplectic invariants and Hamiltonian dynamics},
	Birkh\"auser Verlag, Basel (1994).
\bibitem[IKRT]{IKRT} R. Ib\'a\~nez, J. K\c{e}dra, Yu.\ Rudyak and A. Tralle,
	{\it On fundamental groups of symplectically aspherical manifolds},
	Math.\ Z. \textbf{248} (2004), no.~4, 805--826.
\bibitem[KRT]{KRT} J. K\c{e}dra, Yu.\ Rudyak and A. Tralle,
	{\it On fundamental groups of symplectically aspherical manifolds II: Abelian groups},
	Math.\ Z. \textbf{256} (2007), no.~4, 825--835.
\bibitem[Ko]{Ko} P. Kontorovi\v{c},
	{\it Groups with a basis of partition III},
	Mat.\ Sbornik N.S. \textbf{22} (64) (1948), 79--100. 
\bibitem[Ku]{Ku} A. Kurosh,
	{\it The theory of groups},
	Translated from the Russian and edited by K. A. Hirsch,
	2nd English ed.~2 volumes Chelsea Publishing Co., New York (1960).
\bibitem[LT]{LT} G. Liu and G. Tian,
	{\it Floer homology and Arnold conjecture},
	J. Differential Geom.\ \textbf{49} (1998), 1--74.
\bibitem[LO]{LO} G. Lupton and J. Oprea,
	{\it Cohomologically symplectic spaces: Toral actions and the Gottlieb group},
	Trans.\ Amer.\ Math.\ Soc.\ \textbf{347} (1995), 261--288.
\bibitem[Ma]{Ma} M. Mazzucchelli,
	{\it  Symplectically degenerate maxima via generating functions},
	Math.\ Z. \textbf{275} (2013), no.~3--4, 715--739.
\bibitem[On]{On} K. Ono,
	{\it Floer-Novikov cohomology and the flux conjecture},
	Geom.\ Funct.\ Anal.\ \textbf{16} (2006), no.~5, 981--1020.
\bibitem[Or]{Or} R. Orita,
	{\it Non-contractible periodic orbits in Hamiltonian dynamics on tori},
	Bull.\ Lond.\ Math.\ Soc.\ \textbf{49} (2017), no.~4, 571--580.
\bibitem[Pa]{Pa} J. Pardon,
	{\it An algebraic approach to virtual fundamental cycles on moduli spaces of pseudo-holomorphic curves},
	Geom.\ Topol.\ \textbf{20} (2016), 779--1034.
\bibitem[RT]{RT} Y. Rudyak and A. Tralle,
	{\it On symplectic manifolds with aspherical symplectic form},
	Topol.\ Methods Nonlinear Anal.\ \textbf{14} (1999), 353--362.
\bibitem[Sa]{Sa} D. Salamon,
	{\it Lectures on Floer homology},
	in Symplectic Geometry and Topology (Park City, Utah, 1997), IAS/Park City Math.\ Ser.\ \textbf{7}, Amer.\ Math.\ Soc., Providence (1999), 143--230.
\bibitem[SZ]{SZ} D. Salamon and E. Zehnder,
	{\it Morse theory for periodic solutions of Hamiltonian systems and the Maslov index},
	Comm.\ Pure Appl.\ Math.\ \textbf{45} (1992), 1303--1360.
\bibitem[SS]{SS} M. Shub and D. Sullivan,
	{\it A remark on the Lefschetz fixed point formula for differentiable maps}
	Topology \textbf{13} (1974), 189--191.
\end{thebibliography}

\end{document}